\documentclass[10pt]{article}
\usepackage{amssymb}
\usepackage{graphicx}
\usepackage{xcolor} 
\usepackage{fullpage} 
\usepackage{amsmath}
\usepackage{amsthm}
\usepackage{verbatim}
\usepackage{enumitem}
\setlist[enumerate]{leftmargin=1.5em}
\setlist[itemize]{leftmargin=1.5em}
\definecolor{green}{rgb}{0,0.8,0} 

\newtheorem{theorem}{Theorem}[section]

\newtheorem{lemma}[theorem]{Lemma}

\theoremstyle{definition}

\theoremstyle{remark}
\newtheorem{remark}[theorem]{Remark}

\numberwithin{equation}{section}
\newcommand{\nrm}[1]{\Vert#1\Vert}

\newcommand{\brk}[1]{\langle#1\rangle}

\newcommand{\nnrm}[1]{{\vert\kern-0.25ex\vert\kern-0.25ex\vert #1 
    \vert\kern-0.25ex\vert\kern-0.25ex\vert}}

\newcommand{\lap}{\Delta}

\newcommand{\ud}{\mathrm{d}}
\newcommand{\rd}{\partial}
\newcommand{\nb}{\nabla}

\newcommand{\tell}{\tilde{\ell}}
\newcommand{\bell}{\bar{\ell}}
\newcommand{\alp}{\alpha}

\newcommand{\gmm}{\gamma}

\newcommand{\dlt}{\delta}

\newcommand{\eps}{\epsilon}

\newcommand{\kpp}{\kappa}

\newcommand{\omg}{\omega}

\newcommand{\bbR}{\mathbb R}

\newcommand{\bbT}{\mathbb T}

\newcommand{\bbZ}{\mathbb Z}

\newcommand{\calD}{\mathcal D}
\newcommand{\calE}{\mathcal E}

\newcommand{\calL}{\mathcal L}

\newcommand{\calS}{\mathcal S}

\begin{document}

\title{Vortex stretching and enhanced dissipation for the incompressible 3D Navier-Stokes equations} 
\author{In-Jee Jeong\thanks{School of Mathematics, Korea Institute for Advanced Study. E-mail: ijeong@kias.re.kr} \and Tsuyoshi Yoneda\thanks{Department of Mathematics, University of Tokyo. E-mail: yoneda@ms.u-tokyo.ac.jp}}
\date{\today}

 


\maketitle


\begin{abstract}
	We consider the 3D incompressible Navier-Stokes equations under the following $2+\frac{1}{2}$-dimensional situation: small-scale horizontal vortex blob being stretched by large-scale, anti-parallel pairs of vertical vortex tubes. We prove enhanced dissipation induced by such vortex-stretching. 
\end{abstract}


\section{Introduction}
 
The zeroth law of turbulence states that, in the limit of vanishing viscosity, the rate of kinetic energy dissipation for solutions to the incompressible Navier-Stokes equations becomes nonzero. This is one of the central ansatz of Kolmogorov's 1941 theory (\cite{K41}). To formulate this law, we recall the 3D incompressible Navier-Stokes equations on $\mathbb{T}^3_*:=(\mathbb{R}/2\mathbb{Z})^3$: \begin{equation} \label{NS}
\left\{
\begin{aligned}
\rd_t u^\nu + u^\nu\cdot\nabla u^\nu+\nabla p = \nu \lap u^\nu + f,\\
\nabla\cdot u^\nu = 0,\\
u^\nu(t=0) = u^\nu_0
\end{aligned}
\right.
\end{equation} where $\nu>0$ is the viscosity and $u^\nu : \bbT^3_* \rightarrow \bbR^3, p : \bbT^3_* \rightarrow \bbR$ denote the velocity and pressure of the fluid, respectively. Here $f : \bbT^3_* \rightarrow \bbR^3$ is some external force. Assuming that the solution is sufficiently smooth, taking the dot product of the equation with $u^\nu$ and integrating over $\bbT^3_*$ gives the energy balance \begin{equation*}
\begin{split}
\frac{d}{dt} \frac{1}{2}\nrm{u^\nu(t)}_{L^2}^2 = \int_{\bbT^3_*} f(t) \cdot u^\nu(t) dx - \nu \nrm{\nabla u^\nu(t)}_{L^2}^2 . 
\end{split}
\end{equation*} The zeroth law then postulates that, under the normalization $\nrm{u_0^\nu}_{L^2} = 1$, the mean energy dissipation rate does not vanish as $\nu \rightarrow 0^+$: \begin{equation*}
\begin{split}
\liminf_{\nu\rightarrow 0}\nu\brk{ |\nabla u^\nu |^2 }> 0,
\end{split}
\end{equation*} where $\brk{\cdot}$ usually denotes some \textit{ensemble} or long-time, space averages. 
Laboratory experiments and numerical simulations of turbulence both confirm the above zeroth law (\cite{BV,E3,KIYIU,Va}).  See recent works of Drivas \cite{Dr} and Buckmaster-Vicol \cite{BV} for more precise formulation and developments related to  the zeroth law. 
In this paper, we take $\brk{\cdot}$ to be a short-time space average, and take sequences of smooth initial data $u_0^{\nu} \in C^\infty(\bbT^3_*)$. 
Hence we may take $f \equiv 0$, and the energy balance is justified. 
However, in the short-time, a trivial version of zeroth law appears, and thus we need to avoid it carefully. 
We now explain it more precisely. {Let $H^s$ ($s\in\mathbb{R}$) be Sobolev spaces.} If we choose $\{u_{0,n}\}_n$ satisfying $\|u_{0,n}\|_{H^1}\to\infty$,
and choose $T_n$ $(T_n\to 0, n\to\infty$) to be 
$\sup_{0<t<T_n}\nu_n\|u_{0,n}-u_{0,n}^{\nu_n}(t)\|^2_{H^1}<\epsilon$ (for sufficiently small $\epsilon>0$) with  $\nu_n\approx\|u_{0,n}\|_{H^1}^{-2}$,
then we have 
\begin{equation}
\lim_{n\to \infty}\, \nu_n \frac{1}{T_n}\int_0^{T_n}\int_{\mathbb{T}^2}|\nabla u^{\nu_n}_n(t,x)|^2dxdt \approx 1.
\end{equation}
This is simply due to the fact that  $\nu_n\approx \|u_{0,n}\|_{H^1}^{-2}$. 
Thus to consider a non-trivial zeroth law in mathematics, it is necessary to add the following condition:
\begin{equation}\label{vortex-stretching-condition}
\liminf_{n\to \infty}\, \frac{\frac{1}{T_n}\int_0^{T_n}\int_{\mathbb{T}^2}|\nabla u^{\nu_n}_n(t,x)|^2dxdt}{\|\nabla u_{0,n}\|_{L^2}^2}=\infty.
\end{equation}
The above condition can be interpreted as occurrence of strong ``vortex-stretching". We shall prove a version of the zeroth law satisfying the above, which implies in particular enhanced dissipation. 
 We achieve this in the framework of $2+\frac{1}{2}$-dimensional flow, which we now explain.

\subsection{The $2+\frac{1}{2}$-dimensional flow}

The incompressible Euler equations are obtained by taking $\nu = 0$ in \eqref{NS}. Introducing the vorticity $\omega = \nabla \times u$, we obtain the 3D vorticity equations: 
\begin{equation*}
\begin{split}
\partial_t\omega + (u\cdot\nabla) \omega = (\omega\cdot \nabla) u,\quad x \in \mathbb{T}^3_* := (\mathbb{R}/2\mathbb{Z})^3 
\end{split}
\end{equation*} where the velocity $u$ is determined by the (periodic) 3D Biot-Savart law: \begin{equation*}
\begin{split}
u(t,x) = \int_{\mathbb{T}^3_*} K_3 \large( x - y \large) \omega(t,y) \, dy, 
\end{split}
\end{equation*} with \begin{equation*}
\begin{split}
K_3(x)v  = \frac{1}{4\pi} \frac{x \times v}{|x|^3} \quad (\text{with reflections}).
\end{split}
\end{equation*} The associated Lagrangian flow is then given by \begin{equation*}
\begin{split}
\partial_t \Phi(t,x)=u(t,\Phi(t,x))\quad\text{with}\quad \Phi(0,x)=x\in\mathbb{T}^3_*.
\end{split}
\end{equation*} 
In this paper, we shall examine a sequence of smooth initial vorticity of the form $$\omega_{n,0} = \omega_{n,0}^\calL + \omega_{n,0}^\calS$$ and we  restrict them to the following symmetry (with a slight abuse of notation):
\begin{equation*}
\omega^\calL_{n,0}=(0,0,\omega_{n,0}^\calL(x_1,x_2))^T\quad
\text{and}\quad 
\omega_{n,0}^\calS=(\omega_{n,0,1}^\calS(x_1,x_2), \omega_{n,0,2}^\calS(x_1,x_2),0)^T. 
\end{equation*} 
The corresponding solution also keeps this symmetry, which is commonly referred as to the $2+\frac{1}{2}$-dimensional flow. Note that the data and solution are independent of $x_3$, and in this setting there is a global unique smooth solution to the 3D Euler equations (also for the 3D Navier-Stokes) with initial data $\omg_{n,0}$, which we shall denote by $\omega_{n}(t)$. By the Biot-Savart law,
\begin{eqnarray*}
	\nonumber
	u_{n}(t,x)
	= 
	\int_{\mathbb{T}^3_*} K_3\big(x - y \big) \omega_{n}(t,y) \, dy,\
	\partial_t \Phi_{n}(t,x)=u_{n}(t,\Phi_{n}(t,x))
\end{eqnarray*} 
and then
\begin{equation*} 
\omega_{n}(t,\Phi_{n}(t,x)) 
= 
D\Phi_{n}(t,x) \omega_{n,0}(x)
=D\Phi_{n}(t,x)(\omega^\calL_{n,0}(x)+\omega_{n,0}^\calS(x)),
\end{equation*} 
where $D\Phi_n = (\rd_j\Phi_{n,i})_{1\le i,j \le 3}$. This is the famous Cauchy formula. 
Moreover, since there is no dependence on the third variable for the solution $u$, $\Phi_n$ is determined by the 2D flow arising from the solution of the 2D Euler equations with initial data $\omg^\calL_{n,0}(x_1,x_2)$. We denote the 2D flow map by $\eta_n$, and by trivially extending the 2D flow into 3D, with some abuse of notation, the 3D flow map associated with the solution for $\omega^\calL_{0,n}$ can be written as 
\begin{equation*}
D\eta_n:=
\begin{pmatrix}
\partial_1\eta_{n,1}&\partial_2\eta_{n,1}&0\\
\partial_1\eta_{n,2}&\partial_2\eta_{n,2}&0\\
0&0&1
\end{pmatrix},
\quad
\quad
D\eta_n^{-1}=
\begin{pmatrix}
\partial_2\eta_{n,2}&-\partial_2\eta_{n,1}&0\\
-\partial_1\eta_{n,2}&\partial_1\eta_{n,1}&0\\
0&0&1
\end{pmatrix}.
\end{equation*}
It is not difficult to verify that $\rd_j\Phi_{n,i} = \rd_j\eta_{n,i}$ for $1 \le i, j \le 2$ and we have the following explicit formulas: 
\begin{equation}\label{scale-separation-expression}
\begin{split}
\omega_n^\calL(t,\eta_n(t,x)) = \omega_{n,0}^\calL(x)\quad\text{and}\quad \omega_n^\calS(t,\eta_n(t,x)) = D\eta_n(t,x)\omega_{n,0}^\calS(x).
\end{split}
\end{equation}
Again, by the Biot-Savart law, we can also recover the large-scale velocity:
\begin{equation*}
\nonumber
u_{n}^\calL(t,x)
= 
\int_{\mathbb{T}^3_*} K_2\big(x - y \big) \omega_{n}^\calL(t,y) \, dy
\end{equation*}  where \begin{equation*}
\begin{split}
K_2(x) = \frac{1}{2\pi} \frac{x^\perp}{|x|^2} \quad \mbox{(with reflections)} 
\end{split}
\end{equation*}
and also $u_{n}^\calS(t,x)
=
u_{n,0}^\calS(t,\eta_n(t,x))$ where $u_{n}^\calS(t,x)$ can be uniquely recovered from $\omg_{n}^\calS(t,x)$ by $\nabla\times u_{n,0}^\calS=\omega_{n,0}^\calS$ and $\nb\cdot u_{n}^\calS(t,x) = 0$. Now note that since 
$\omega^\calS_{n}(t)\perp e_3$ and $\omega^\calL_{n}(t)\parallel e_3$,  
we have $$\|\omega_n(t)\|_{L^2}^2=\|\omega^\calL_{n}(t)\|_{L^2}^2+\|\omega^\calS_{n}(t)\|_{L^2}^2.$$

\subsection{Main results}

To state our result, let us briefly explain the construction of the initial data sequence. We consider data independent of $x_3$, which allows us to treat them as functions defined on $\bbT^2_L := (\bbR/(2L\bbZ))^2$. We take the Bahouri-Chemin stationary solution introduced in \cite{BC} $\omega(x_1,x_2) = \mathrm{sgn}(x_1)\mathrm{sgn}(x_2)$ on $[-L,L]^2$ and smooth it out at scale $\ell \ll L$ to define $\omega^{\calL}_{n,0}(x)$. Then we place a small ``bump'' $\omega^\calS_{n,0}(x)$ in a ball of radius $\tell \ll \ell$ centered at the origin. Then the initial data sequence is simply given by $\omega_{n,0} = \omega^{\calL}_{n,0} + \omega^{\calS}_{n,0}$, where $\ell, \tell$ (and even $L$ in some cases) depend on $n$. Details of the construction will be explained in Section \ref{sec:proof}; for now see Figure \ref{fig:setup}.
	We now give the main theorem, which roughly state that the vortex-stretching in the $2+\frac{1}{2}$-dimensional setup is enough to create vortex stretching of order $\nu^{-\eps_0}$ for some $\eps_0>0$ for the 3D Navier Stokes equations in the limit $\nu \rightarrow 0^+$, with uniformly bounded {(at least) in $L^2$ initial data}. To motivate the statements, let us recall the energy identity for the Navier-Stokes equations: \begin{equation*}
	\begin{split}
	\frac{1}{2}\nrm{u^\nu(t)}_{L^2}^2 + \nu\int_0^t \nrm{\nb u^\nu(s)}_{L^2}^2 ds = \frac{1}{2}\nrm{u^\nu_0}_{L^2}^2. 
	\end{split}
	\end{equation*} From the symmetry in our initial data, the solution $u^\nu(t)$ can be written as the sum $u^{\calL,\nu} +u^{\calS,\nu}$ where $u^{\calL,\nu} $ and $u^{\calS,\nu}$ are defined by \begin{equation} 
	\left\{
	\begin{aligned} 
	&\rd_t u^{\calL,\nu} + u^{\calL,\nu} \cdot \nb u^{\calL,\nu} + \nb p^{\nu}= \nu \lap u^{ {\calL},\nu} , \quad \nb \cdot u^{\calL,\nu} = 0, \\
	&\rd_t u^{\calS,\nu} + u^{\calL,\nu} \cdot \nb u^{\calS,\nu} = \nu \lap u^{\calS,\nu} 
	\end{aligned}
	\right.
	\end{equation}  and we similarly have the following energy identity for the small-scale: \begin{equation*}
	\begin{split}
	\frac{1}{2}\nrm{u^{\calS,\nu}(t)}_{L^2}^2 + \nu\int_0^t \nrm{\nb u^{\calS,\nu}(s)}_{L^2}^2 ds = \frac{1}{2}\nrm{u^{\calS,\nu}_0}_{L^2}^2. 
	\end{split}
	\end{equation*}

 {\begin{theorem}[Enhanced dissipation] \label{thm:main4} There exists some absolute constant $\dlt>0$  such that the following statements hold: for any $0<\bar{a}_0<1$, there exist $L_n\le 1$, viscosity constants $\nu_n \rightarrow 0$, and a sequence of $C^\infty$-smooth initial data $\omg_{n,0} = \omg_{n,0}^\calL+ \omg_{n,0}^\calS$  with uniform bounds \begin{equation*}
	\begin{split}
	 \|u_{n,0}\|_{L^2(\mathbb{T}^3_n)}^2=\nrm{u_{n,0}^\calL}^2_{L^2(\bbT^3_n)} + \nrm{u_{n,0}^\calS}^2_{L^2(\bbT^3_n)}{\approx L_n^2,}
	\end{split}
	\end{equation*} \begin{equation*}
	\begin{split}
	\nrm{\omg_{n,0}^\calL}_{L^\infty(\bbT^3_n)} \lesssim 1, 
	\end{split}
	\end{equation*} defined on the torus $\bbT^3_n := (\bbR/(2L_n\bbZ))^2\times(\bbR/(2\bbZ))$ such that the unique smooth solution $u^{\nu_n}_{n}$ of the 3D Navier-Stokes equations with initial data $u_{n,0}$ and viscosity $\nu_n$ on $\bbT^3_n$ satisfies \begin{equation}\label{eq:modified-1}
	\begin{split}
	\liminf_{n \rightarrow \infty} \nu_n^{\bar a_0} \frac{1}{\dlt}\int_0^{\dlt} \nrm{\nb u^{\nu_n}_n(t)}_{L^2(\bbT_n^3)}^2 dt \gtrsim \|u_{n,0}\|_{L^2(\mathbb{T}^3_n)}^2.
	\end{split}
	\end{equation}  
\end{theorem}}

\begin{remark} We  remark that for $\bar{a}_0\le \frac{1}{2}$, we can take $L_n = 1$ for all $n$, while for $\bar{a}_0>\frac{1}{2}$, we need $L_n \rightarrow 0$. Of course, one can take $L_n$ to be dyadic and still regard the data as being defined on the unit torus; see Figure \ref{fig:setup3} illustrating this point. 
\end{remark}

\subsection{Discussions}

\subsubsection{Recent theoretical developments and ideas of the proof}

Regarding the well-posedness theory of the incompressible Euler equations, a recent breakthrough was made in the work of Bourgain-Li \cite{BL} (see also \cite{EJ,EM,MY,KS}) where the authors have shown \textit{ill-posedness} of the Euler equations in critical Sobolev spaces. In the case of 2D, the critical $L^2$-based Sobolev space is $H^1$ in terms of the vorticity. The strategy in \cite{BL} is to show that there exists \textit{large Lagrangian deformation} for arbitrarily short time with initial vorticity uniformly bounded in $H^1$. This large Lagrangian deformation is responsible for the statement of Theorem \ref{thm:main4}, as the small-scale vorticity is being stretched by the deformation of the base large-scale flow. To achieve this we need to prove a sharp and quantitative bounds on the Lagrangian deformation, using smoothed-out Bahouri-Chemin solutions. This should be compared with previous results \cite{BL,EJ} where Lagrangian deformation and vorticity norm growth were obtained via a contradiction argument. 

Another important breakthrough regarding the 2D Euler equations was the work of Kiselev-Sverak \cite{KS} on the double exponential growth of the vorticity gradient. The main tool was the so-called ``Key Lemma'' which surprisingly gave an explicit integral representation for the main term in the velocity gradient for vorticity capped in $L^\infty$ and is odd with respect to both axes (i.e. anti-parallel). To calculate in a sharp way the velocity gradient in our setting, we adopt the Kiselev-Sverak approach, which then yields a quantitative large Lagrangian deformation with a careful ODE argument. We achieve this improvement only in the concrete setting of perturbed Bahouri-Chemin vorticities. 

 {We use the $2+\frac{1}{2}$-dimensional flow construction to lift Lagrangian deformation into three dimensional space where vortex stretching is created. This gives a large growth of the $H^1$-norm of the vorticity of the Euler solution. If the Navier-Stokes solutions converge to the Euler solution in $H^1$, this would imply enhanced dissipation for the sequence of Navier-Stokes solutions. To obtain quantitative convergence rate, we perform hard calculations which is the content of Section \ref{subsec:inviscid}. For these calculations we need sharp estimates for the Euler and Navier-Stokes solutions, which are established in Sections \ref{subsec:large} and \ref{subsec:small}. In short, the following are technical advances achieved in this work}: \begin{itemize}
	\item We prove sharp, quantitative bounds on the perturbed Bahouri-Chemin solutions. 
	\item We obtain quantitative inviscid limit estimates, which does not seem available in the literature. 

\end{itemize}

In the estimates we prove in Section \ref{sec:proof}, we have retained all physical parameters until the very end (before \ref{subsubsec:final}), and therefore the resulting estimate could be useful for the readers who would like to try out different scaling of physical quantities as $n \rightarrow +\infty$. 

\subsubsection{Recent numerical results} 
	Let us mention a recent numerical simulation which have inspired the current work. Recently, using direct numerical simulations of the 3D Navier-Stokes equations, Goto, Saito, and Kawahara \cite{GSK} have found that sustained turbulence consists of a hierarchy of antiparallel pairs of vortex tubes. Their main conclusions can be summarized as follows, which bear some similarity with our constructions:  
	\begin{itemize}
		\item Turbulence, in the inertial length scales, is composed of hierarchy of vortex tubes with different sizes.
		\item At each hierarchy level, vortex tubes tend to form antiparallel pairs and they effectively
		stretch and create smaller-scale vortex tubes. Moreover, stretched vortex tubes tend to align in the direction perpendicular to larger-scale vortex tubes.
		\item Vortices at each hierarchical level are 
		most likely to be stretched in strain fields around $2$-$8$ times larger vortices.
	\end{itemize}
	It would be interesting to push our results further to be closer to the picture they have.

 {
\subsubsection{Energy dissipation for solutions with one-point singularity}\label{subsubsec:energy-diss}
} In our result, the large-scale vorticity is uniformly bounded in $L^\infty$. Therefore it is tempting to approach the actual zeroth law using initial data which is singular, e.g. vorticities which are only $C^{ {\alp-1}}$ and not better for $0\le\alp<1$. Regarding this point, we present a simple computation which illustrates that,  {when one considers velocity fields which is $C^\alp$ at a single point (say at the origin) and smooth away from it, the nonlinearity is not strong enough to cause anomalous energy dissipation.\footnote{This computation was suggested to us by one of the referees.} To this end, consider $\overline{u}_{\ell} := \varphi_{\ell}*u$ ($\varphi_{\ell}=\ell^{-d}\varphi(\ell^{-1}\cdot)$ with some mollifier $\varphi$) and we compute in $d$-dimensions the instantaneous energy change at $t = 0$: \begin{equation*}
		\begin{split}
		\frac{1}{2} \left.\frac{d}{dt}\right|_{t=0} \nrm{\overline{u}_{\ell}}_{L^2}^2 = \int_{\bbT^d} \nb\overline{(u_0)}_{\ell} : \tau_{\ell}(u_0,u_0) \ud x, 
		\end{split}
		\end{equation*} where \begin{equation*}
		\begin{split}
		\tau_{\ell}(u_0,u_0) = \overline{(u_0\otimes u_0)}_{\ell} - \overline{(u_0)}_{\ell}\otimes\overline{(u_0)}_{\ell}. 
		\end{split}
		\end{equation*} Splitting the integral into the ball around the singularity $B_\eps(0)$ and the rest $\bbT^d\backslash B_\eps(0)$, we obtain (cf. \cite{CET}) \begin{equation*}
		\begin{split}
		\left| \frac{1}{2} \left.\frac{d}{dt}\right|_{t=0} \nrm{\overline{u}_{\ell}}_{L^2}^2  \right| \lesssim \nrm{u_0}^3_{C^\alp(\bbT^d)} \ell^{3\alp-1}\eps^{d} + \nrm{u_0}^3_{C^1(\bbT^d\backslash B_\eps(0))}\ell^2 \lesssim \ell^{3\alp-1}\eps^d + \eps^{3(h-1)}\ell^2
		\end{split}
		\end{equation*} where we have used $\nrm{u_0}_{C^1(\bbT^d\backslash B_\eps(0))}\lesssim \eps^{\alp-1}$. Two expressions can be balanced by setting \begin{equation*}
		\begin{split}
		\eps = \ell^{\gmm_{\alp,d}}, \quad \gmm_{\alp,d} = \frac{3(1-\alp)}{3(1-\alp)+d} \in (0,1),
		\end{split}
		\end{equation*} and this gives the bound \begin{equation*}
		\begin{split}
		\left| \frac{1}{2} \left.\frac{d}{dt}\right|_{t=0} \nrm{\overline{u}_{\ell}}_{L^2}^2  \right| \lesssim\ell^{3\alp-1 + \gmm_{\alp,d}d }. 
		\end{split}
		\end{equation*} One sees that the exponent \begin{equation*}
		\begin{split}
		3\alp-1 + \gmm_{\alp,d}d  = \frac{9\alp(1-\alp) + 3\alp + 2d-3}{3(1-\alp)+d}>0
		\end{split}
		\end{equation*} whenever $0\le h<1$ and $d\ge 2$. This calculation (which in particular incorporates the case of $2+\frac{1}{2}$-dimensional flows) shows that the instantaneous energy change vanishes with the rate $\beta:= 3\alp-1 + \gmm_{\alp,d}d$.}

We refer to recent works of Luo and Shvydkoy (\cite{LS1, LS2, Sh}) which systematically studies the radially homogeneous solutions to 2D and 3D Euler equations and conclude absence of anomalous dissipation in that class of solutions. \\

 {\subsubsection{Upper bound on energy dissipation in the vanishing viscosity limit} 
	Given a sequence of initial data (normalized in $L^2$ norm by $L_n^2$) and viscosity constants, it is reasonable to define the index $0\le \bar{b}_0<1$ \begin{equation*}
	\begin{split}
	\bar{b}_0 := \inf\{ b_0 : \limsup_{n\rightarrow\infty} \nu_n^{b_0} \nrm{\nb u_{n,0}}_{L^2(\bbT^3_n)}^2 \lesssim \nrm{u_{n,0}}_{L^2(\bbT^3_n)}^2\}.
	\end{split}
	\end{equation*} In the above theorem, one can check from the proof {(see \eqref{avoid-trivial-zeroth-law})} that \begin{equation*}
	\begin{split}
	\bar{b}_0 = \bar{a}_0 - c_*\dlt, 
	\end{split}
	\end{equation*}  where $c_*>0$ is a constant depending only on $\bar{a}_0$ which possibly vanishes only when $\bar{a}_0 \rightarrow 1$ {(this $\bar b_0$ consideration essentially comes from \eqref{vortex-stretching-condition})}. 
	On the other hand, if one is interested only in the case of $\bar{b}_0=0$, we can take $\bar{a}_0 = c_0\dlt$ where $c_0>0$ is an absolute constant, with initial data sequence $\{ u_{n,0} \}$ uniformly bounded in $H^1(\bbT^3_*)$ with $\bbT^3_* = (\bbR/(2\bbZ))^3$. In this case, a recent result of Drivas and Eyink \cite{DE} puts a restriction that $c_0\dlt<\frac{2}{3}$, where $\dlt>0$ is the same universal constant in the statement of Theorem \ref{thm:main4}.   Let us explain it more precisely.
		They showed that if a sequence of Leray solutions $\{u^\nu\}_\nu$ are uniformly bounded in
		$L^3([0,\delta];B^\sigma_{3,\infty}(\mathbb{T}^3_*))$ for some $\sigma\in (0,1)$,
		then the corresponding solutions satisfy
		\begin{equation}\label{dissipation estimate}
		\nu \frac{1}{\dlt}\int_0^\delta\int_{\mathbb{T}^3_*} |\nb u^\nu(t,x)|^2  dxdt\lesssim \nu^{\frac{3\sigma-1}{\sigma+1}}.
		\end{equation}
		(Note that the function space $L^3([0,\delta];B^\sigma_{3,\infty}(\mathbb{T}^3_*))$ is physically natural; see Remark 1 in \cite{DE}.) The estimate \eqref{dissipation estimate} gives an upper bound on the value of the constant $\bar a_0$ from \eqref{eq:modified-1}: for $\sigma> (2-\bar a_0)/(2+\bar a_0)$, the sequence of solutions $\{u_n^{\nu_n}\}_n$ (the corresponding vorticities are $\{\omega_n^{\nu_n}\}_n$) does not belong to $L^3([0,\delta];B^\sigma_{3,\infty}(\mathbb{T}^3_*))$  uniformly in $n$.
		The proof is the following: 
		assume to the contrary that the sequence of solutions $\{u_n^{\nu_n}\}_n$ belongs to $L^3([0,\delta]; B^\sigma_{3,\infty}(\mathbb{T}^3_*))$ uniformly in $n$. 
		By \eqref{eq:modified-1}, we see
		\begin{equation*}
		\nu_n\int_0^\delta\int_{\mathbb{T}^3_*}|\nabla u^{\nu_n}_n(t,x)|^2dxdt\gtrsim\nu_n^{1-\bar a_0}.
		\end{equation*}
		Thus, if $\sigma$ satisfies  $1-\bar a_0>\frac{3\sigma-1}{\sigma+1}$, that is,
		$\sigma> (2-\bar a_0)/(2+\bar a_0)$, then this contradicts \eqref{dissipation estimate} for sufficiently large $n$. On the other hand,  the sequence of solutions $\{u^{\nu_n}_n\}_n$  belongs uniformly in $L^3_t B^\sigma_{3,\infty}$ with some 
		$\sigma$. To see this, one can directly estimate the equation \begin{equation*} 
		\begin{split}
		&\rd_t \omega_n^{\mathcal S,\nu} + u_n^{\mathcal L,\nu} \cdot \nabla \omega_n^{\mathcal S,\nu} = \nabla u_n^{\mathcal L,\nu} \omega_n^{\mathcal S,\nu} + \nu \lap \omega_n^{\mathcal S,\nu}
		\end{split}
		\end{equation*} in $L^p$: $\nrm{\omega_n^{\mathcal S,\nu}(t)}_{L^p} \lesssim \nrm{\omega_{n,0}^{\mathcal S}}_{L^p} \exp(\int_0^t \nrm{\nabla u_n^{\mathcal L,\nu}(s)}_{L^\infty} ds)$, with an implicit constant independent of $\nu \ge 0$. From our choice of initial data and $\nrm{\nabla u_n^{\mathcal L,\nu}}_{L^\infty} \lesssim n$ (see Lemma \ref{lem:velgrad-global} for details), it follows that the corresponding solution $\omega_n^{\mathcal S,\nu}$ belongs  to $L^\infty([0,t];L^{p(t)}(\mathbb{T}_*^2))$ with $p(t) = 2-ct$ for some constant $c>0$. 
		This is due to the fact that $\|\omega_{n,0}^{\mathcal S}\|_{L^p}e^{\int_0^t\|\nabla  {u}_n^{\mathcal L,\nu}(s)\|_{L^\infty}ds}\lesssim e^{n(1-2/p+t)}$
		and to get the uniform bound, $1-2/p+t$ must be zero.
		Then at least for $t > 0$ sufficiently small, the velocity must be uniformly in $L^\infty_t W^{1,p(t)} \subset L^3_t B^\sigma_{3,\infty}$ with $2-3/p(t) = \sigma$.  This gives the restriction that $\bar a_0 \le 2/3$.

}

\begin{figure}\includegraphics[scale=0.8]{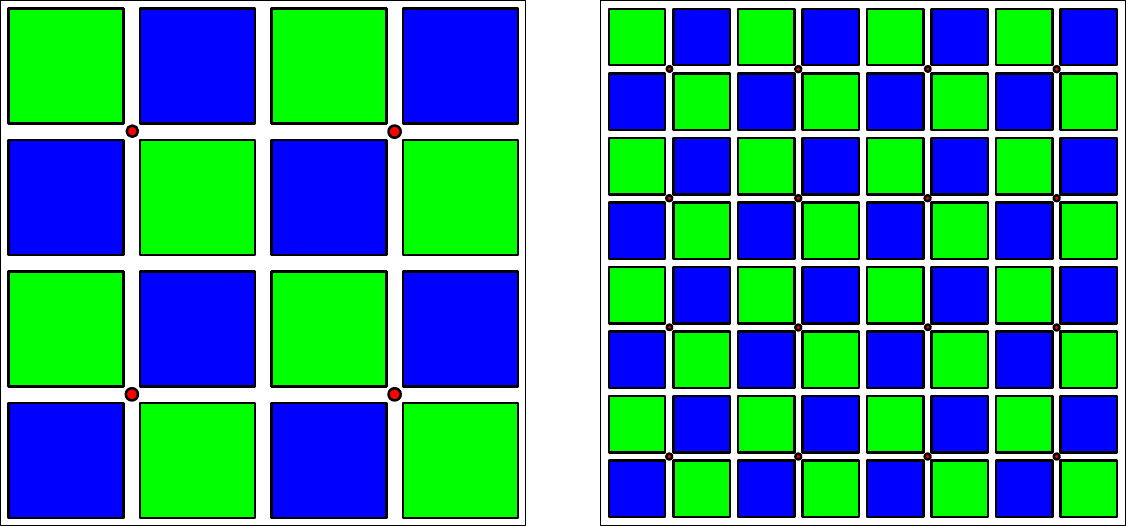} \centering\caption{Data used in the case $a_0>\frac{1}{2}$ of Theorem \ref{thm:main4} (depicted on the unit torus): blue and green represent the regions of positive and negative large-scale vorticity, and red dots represent the support of small-scale vorticity.}\label{fig:setup3}\end{figure}

\subsection{Organization of the paper} The rest of this paper is organized as follows: we first collect the notations and conventions that we use. The entire Section \ref{sec:proof} is devoted to the proofs of the main results. In \ref{subsec:large}, we define the (sequence of) large-scale vorticity and obtain various sharp estimates. In particular, we prove creation of large Lagrangian deformation. Then in \ref{subsec:small}, we explain the setup for the (sequence of) small-scale vorticity and establish sharp upper bounds for them. Finally in \ref{subsec:inviscid}, we perform inviscid limit computations and conclude the proof.

\subsection{Notations and parameters} 

\subsubsection{Notations} 
For the reader's convenience, we collect the notations that will be used frequently in the paper. 
\begin{itemize}
	\item We shall work with the 2D domain $\bbT^2_L = (\bbR/(2L\bbZ))^2$ and $\bbT^3 = (\bbR/(2L\bbZ))^2 \times (\bbR/(2\bbZ))$ where $0<L\le 1$.
	\item Given a scalar-valued function $f : \bbT^2_L \rightarrow \bbR$, we define the $L^p$ norms by \begin{equation*} 
	\begin{split}
	& \nrm{f}_{L^p}^p = \int_{\bbT^2_L} |f|^p dx,\quad 1 \le p < +\infty. 
	\end{split}
	\end{equation*} The case $p = +\infty $ is given by \begin{equation*} 
	\begin{split}
	& \nrm{f}_{L^\infty} = \mathrm{ess-sup}_{x \in \bbT^2_L} |f(x)|. 
	\end{split}
	\end{equation*} 
	\item If $v$ is a vector-valued function $v = (v_1,\cdots, v_d)$, \begin{equation*} 
	\begin{split}
	& \nrm{v}_{L^p}^p := \sum_{i=1}^d \nrm{v_i}_{L^p}^p,\quad 1 \le p < +\infty, \quad \nrm{v}_{L^\infty} := \max_i \nrm{v_i}_{L^\infty}.
	\end{split}
	\end{equation*} 
	\item The homogeneous Sobolev spaces are defined by \begin{equation*} 
	\begin{split}
	& \nrm{v}_{\dot{H}^m} = \nrm{\nabla^mv}_{L^2}
	\end{split}
	\end{equation*} for integers $m \ge 1$, where $\nabla^m v$ is a vector consisting of all possible $m$-th order partial derivatives of $v$. 
	\item The homogeneous H\"older norms are defined by \begin{equation*}
	\begin{split}
	\nrm{f}_{\dot{C}^\alpha} := \sup_{x \ne x'} \frac{|f(x)-f(x')|}{|x-x'|^\alpha}
	\end{split}
	\end{equation*} for $0 < \alpha \le 1$. 
	\item As it is usual, we use the letters $C, c$ to denote various absolute constants whose value can change from a line to another or even within a single line. 
\end{itemize}

\subsubsection{Parameters}
\begin{itemize}
	\item In this paper, $n \rightarrow +\infty$ is a large parameter. We shall use the notation $A \ll B$ (equivalently, $B \gg A$) if the ratio $A/B$ tends to 0 as $n \rightarrow +\infty$, where $A$ and $B$ are positive expressions depending on $n$. Moreover, we use $A \lesssim B$ (equivalently, $B\gtrsim A$) if there is an absolute constant $C>0$ such that $A \le CB$ uniformly for $n \rightarrow +\infty$. Then, we say $A \approx B$ if $A \lesssim B$ and $B\lesssim A$. Finally, we write $A \simeq B$ if $A/B \rightarrow 1$ as $n \rightarrow +\infty$. 
	\item We shall consider the solutions defined on the time interval $[0,\dlt]$, where we take $\dlt>0$ to be smaller whenever it becomes necessary, without explicitly mentioning it. We emphasize   that $\dlt$ is independent of $n$. 
	\item We comment on a few important parameters: $L$, $\ell$, $\tell$, and $\bell$, all of which depend on $n$. We use $L\le 1$ to denote the length-scale of the torus, which we also take to be the length-scale of the large-scale vorticity. The gradient of the large-scale vorticity is taken to be of order $\ell^{-1}$, where $\ell \ll L$. We introduce the convenient notation $\bell := \ell L^{-1}$, which is a non-dimensional parameter. One may simply fix it as $\bell_n = 2^{-n} \rightarrow 0$. Finally, $\tell := \ell^{1+c\dlt}$ ($c>0$ is some small absolute constant) is the length-scale of the small-scale vorticity.  {All the other parameters are determined using $\bell_n$ and $\ell$; for the case of the viscosity constant $\nu_n$, see \eqref{eq:nu_n}.} 
\end{itemize}

\subsection*{Acknowledgements}
 {The authors sincerely thank the anonymous referees for very helpful comments regarding the manuscript, which have been incorporated in the current paper. We especially thank one of the referees for kindly providing us the calculations \ref{subsubsec:energy-diss}, which clarifies the situation.}

We thank Professors A. Mazzucato and T. Drivas for inspiring communications and telling us about the articles \cite{CD} and \cite{DE}, respectively. We are also grateful to Professors P. Constantin and T. Elgindi for valuable comments.

Research of TY  was partially supported by Grant-in-Aid for Young Scientists A (17H04825), Grant-in-Aid for Scientific Research B (15H03621, 17H02860, 18H01136 and 18H01135), Japan Society for the Promotion of Science (JSPS).  IJ has been supported  by a KIAS Individual Grant MG066202 at Korea Institute for Advanced Study, the Science Fellowship of POSCO TJ Park Foundation, and the National Research Foundation of Korea grant No. 2019R1F1A1058486. 

\section{Proofs}\label{sec:proof}

Before we proceed to the description of the sequences of large and small scale vorticities, which will be denoted by $\omega^\calL_n$ and $\omega^\calS_n$, respectively.

\subsection{Setup for the large-scale vorticity}\label{subsec:large}

\begin{figure}\includegraphics[scale=0.6]{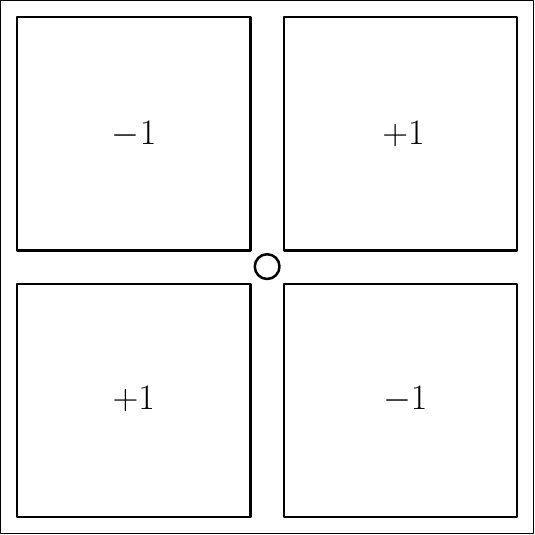} \centering\caption{A diagram showing the support of $\omega^\calL_{n,0}$ (four large squares) and $\omg^\calS_{n,0}$ (circle in the center) in $\bbT^2$.}\label{fig:setup}\end{figure} 

\subsubsection{Estimates for smoothed out Bahouri-Chemin solutions}

Here, we precisely define the smoothed-out Bahouri-Chemin data and prove estimates for the corresponding solutions. For some length-scale $L>0$, we set $\mathbb{T}^2:=\mathbb{R}^2/(2L\mathbb{Z})^2$, and recall that the Bahouri-Chemin solution can be written as $\mathrm{sgn}(x_1)\mathrm{sgn}(x_2)$ where $|x_1|, |x_2| \le L$. Given a length scale $\ell = \ell_n \ll L$, we cut the Bahouri-Chemin solution near the axes as follows: \begin{equation*}
\begin{split}
\tilde{\omega}_n(x_1,x_2) := \mathrm{sgn}(x_1)\mathrm{sgn}(x_2) \chi_{ \{ \ell < |x_1| , |x_2|< L - \ell \} } 
\end{split}
\end{equation*}
Now let $\varphi \in C^\infty_c(\mathbb{R}^2)$ be a standard mollifier;  a radial function whose support is contained in the unit ball. With $\varphi_{\ell}(x) := \ell^{-2}\varphi(\ell^{-1}x)$, we define 
\begin{equation}\label{BC-initial-data}
\omega_{n,0}^{\calL} := \varphi_{\kpp\ell} * \tilde{\omega}_n
\end{equation} for some $0<\kpp \le \frac{1}{2}$. In the following, we shall denote $\omega_n^\calL(t)$ be the unique solutions of the 2D Euler equation defined respectively on $\bbT^2_L$ with initial data $\omega_{n,0}^\calL$.

We now recall a simple estimate of Yudovich (see e.g. \cite{EJ} for a proof): \begin{lemma}
	Let $\omega(t) \in L^\infty([0,\infty): L^\infty(\mathbb{T}^2_L))$ be a solution of the 2D Euler equations, and $\eta(t)$ be the associated flow map. Then for some absolute constant $c > 0$, we have \begin{equation}\label{eq:Yud}
	\begin{split}
	\left(\frac{|x-x'|}{L}\right)^{1+ct\nrm{\omega_0}_{  L^\infty}} \le \frac{|\eta(t,x) - \eta(t,x')|}{L} \le \left(\frac{|x-x'|}{L}\right)^{1-ct\nrm{\omega_0}_{  L^\infty}},
	\end{split}
	\end{equation} for all $0 \le t  $ and $|x-x'| \le L/2$. 
\end{lemma}

We now take a ``small ball'' region \begin{equation}\label{eq:needle} 
\begin{split}
& \calD =  \{ |x| < \tell\}
\end{split}
\end{equation} where $0< \tell \ll \ell$. The following lemma establishes a sharp estimate for the velocity gradient inside this region.


\begin{lemma}\label{lem:nabla-u-local}
	Let $\omega_n(t)$ be the unique solution to the 2D Euler equations with initial data $\omega^\calL_{n,0}$ given in \eqref{BC-initial-data}. We define the corresponding velocity field by $u^\calL_n(t)$. There exists some constant $c>0$ such that {for any $\dlt>0$,} if $\tell$ satisfies $\tell \le c\ell \bell^{ c\dlt}$, then we have \begin{equation}\label{eq:nabla-u-local}
	\begin{split}
	\rd_1u^\calL_{n,1}(t,x) \ge \frac{2}{\pi}\nrm{\omega_0^\calL}_{  L^\infty}\left((1-C\dlt - \eps_n)\ln \frac{1}{\bell}\right),
	\end{split}
	\end{equation} \begin{equation}\label{eq:nabla-u-local2}
	\begin{split}
	\rd_1u_{n,1}^\calL(t,x) \le \frac{2}{\pi}\nrm{\omega_0^\calL}_{  L^\infty}\left((1 + C\dlt + \eps_n)\ln \frac{1}{\bell}\right),  
	\end{split}
	\end{equation}  and \begin{equation}\label{eq:d1u2} 
	\begin{split}
	& |\rd_2 u_{n,1}^\calL(t,x)| + |\rd_1 u_{n,2}^\calL(t,x)| \le C \nrm{\omega^\calL_0}_{L^\infty}
	\end{split}
	\end{equation} for {$$(t,x) \in \left[0,\frac{\delta}{\nrm{\omega_0}_{L^\infty}} \right]\times \calD,$$} where $\eps_n \rightarrow 0$ as $n \rightarrow +\infty$. 
\end{lemma}

\noindent In the proof, we fix some $0 < \ell$ sufficiently smaller than $L$ and omit the indices $\calL$ and $n$.


\begin{proof} 
	We begin by noting that $\omega_{n,0}$ is odd with respect to both axes, $\omega_{n,0} = 1$ on $[(1+\kpp)\ell,L-(1+\kpp)\ell]^2$, and vanishes on $ [0,L]^2 \backslash [(1-\kappa)\ell, L - (1-\kappa)\ell]^2  $. We claim that for small $\delta > 0$, 
\begin{equation}\label{eq:claim}
	\begin{split}
	\omega(t,x) \equiv \nrm{\omega_0}_{  L^\infty} \quad \mbox{on}\quad (t,x) \in \left[0,\frac{\delta}{\nrm{\omega_0}_{L^\infty}} \right]\times \left[ (1+\kpp)\ell \left(\frac{L}{\ell}\right)^{c\delta},\frac{L}{2}  \right]^2. 
	\end{split}
	\end{equation} 
To show this, it suffices to observe that fluid particles starting from $\partial([(1+\kpp)\ell, L-(1+\kpp)\ell]^2 )$ cannot reach the internal square $( (1+\kpp)\ell \left(\frac{L}{\ell}\right)^{c\delta}, L -(1+\kpp)\ell \left(\frac{L}{\ell}\right)^{c\delta}  )^2$ within time $\delta/\nrm{\omega_n}_{L^\infty} $. For this we need to consider four sides of this internal square. We shall only consider the left side, as the other sides can be treated in a similar way. To this end, take a point of the form $x = ((1+\kpp)\ell,a)$ for some $(1+\kpp)\ell \le a \le L - (1+\kpp)\ell$. Setting $x' = (0,a)$ and applying \eqref{eq:Yud}, we obtain \begin{equation*}
\begin{split}
|\eta_1(t,x) - \eta_1(t,x')| \le L \left( \frac{(1+\kpp)\ell}{L} \right)^{1- c\delta }
\end{split}
\end{equation*} for all $0 \le t \le \delta/\nrm{\omega_n}_{L^\infty}$. Since $\eta_1(t,x') = 0$ (by odd symmetry) and $\eta_1(t,x) > 0$ for all $t$, we deduce that $\eta_1(t,x) \le (1+\kpp)\ell \left(\frac{L}{\ell}\right)^{c\delta}$. Applying a similar argument to the other pieces of the boundary, we deduce \eqref{eq:claim}. A completely parallel argument, but instead using the lower bound in \eqref{eq:Yud} rather than the upper bound, gives that \begin{equation}\label{eq:claim2}
\begin{split}
\omega(t,x) \equiv 0 \quad \mbox{on} \quad (t,x) \in \left[0,\frac{\delta}{\nrm{\omega_0}_{L^\infty}} \right]\times    [0,L]^2 \backslash    \left[ (1-\kappa)\ell \left(\frac{L}{\ell}\right)^{-c\delta}, L - (1-\kappa)\ell \left(\frac{L}{\ell}\right)^{-c\delta}  \right]^2. 
\end{split}
\end{equation}

From now on we shall restrict to $t \in [0,\delta/\nrm{\omega_n}_{L^\infty}]$, and recall explicit formulas \begin{equation*}
\begin{split}
\rd_1 u_1(t,x_1,x_2) = \frac{1}{\pi} \int_{\mathbb{R}^2} \frac{(y_1-x_1)(y_2-x_2)}{|y-x|^4} \omega(y) dy = - \rd_2 u_2(t,x_1,x_2) 
\end{split}
\end{equation*}  where we have extended $\omega$ to $\mathbb{R}^2$ by periodicity and the integral is defined in the sense of principal value. Moreover, assuming for simplicity that $(x_1,x_2)$ does not belong to the support of $\omega(t)$, \begin{equation*} 
\begin{split}
\rd_1 u_2(t,x_1,x_2) = \frac{1}{2\pi} \int_{\mathbb{R}^2} \frac{(y_1-x_1)^2-(y_2-x_2)^2}{|y-x|^4} \omega(y) dy = - \rd_2 u_1(t,x_1,x_2). 
\end{split}
\end{equation*} We now take $0 \le x_1, x_2 < \frac{L}{2}$ and observe the uniform bounds \begin{equation*} 
\begin{split}
& \left| \int_{\mathbb{R}^2 \backslash [-L,L]^2 }  \frac{(y_1-x_1)(y_2-x_2)}{|y-x|^4} \omega(y) dy \right| + \left| \int_{\mathbb{R}^2 \backslash [-L,L]^2 }  \frac{(y_1-x_1)^2-(y_2-x_2)^2}{|y-x|^4} \omega(y) dy \right| \lesssim \nrm{\omega_0}_{L^\infty}. 
\end{split}
\end{equation*} (This is elementary but see for instance \cite{Z} for a proof.)   We are ready to prove the claimed estimates. We proceed in several steps: 
	
	\medskip
	
	\noindent \textit{Step 1. Lower bound of $\rd_1u_1$}
	
	\medskip 
	
	\noindent We now estimate $\rd_1u_1(t)$. In view of the previous bound, we restrict the integral to $[-L,L]^2$ and then to $[0,L] \times [0,L]$ owing to the odd symmetry: \begin{equation*} 
	\begin{split}
	&\left| \rd_1 u_1 (t,x)|_{x=0} -\frac{2}{\pi} \int_{ [0,L] \times [0,L] } \frac{ y_1y_2}{(y_1^2 + y_2^2)^2} \omega(t,y) dy \right| \lesssim  \nrm{\omega_0}_{L^\infty},
	\end{split}
	\end{equation*} where the constant is independent of $L$.
Let 
\begin{equation*}
I:=\frac{2}{\pi} \int_{ [0,L] \times [0,L] } \frac{ y_1y_2}{(y_1^2 + y_2^2)^2} \omega(t,y) dy
\end{equation*}
and note that the integrand is non-negative. Using \eqref{eq:claim}, we obtain a simple lower bound on $I$: \begin{equation*}
\begin{split}
I \ge \frac{2}{\pi}\nrm{\omega_0}_{  L^\infty} \int_{ \left[ (1+\kpp)\ell \left(\frac{L}{\ell}\right)^{c\delta},\frac{L}{2}  \right]^2 } \frac{ y_1y_2}{(y_1^2 + y_2^2)^2} dy.
\end{split}
\end{equation*} This immediately gives \begin{equation}\label{eq:d1u1-lower}
	\begin{split}
	\rd_1u_1(t,x)|_{x=0} \ge \frac{2}{\pi}\nrm{\omega_0}_{  L^\infty}\left((1-\eps_n)\ln \frac{1}{\bell}  \right). 
	\end{split}
	\end{equation}

	\medskip
	
	\noindent \textit{Step 2. Upper bound of $\rd_1u_1$ along the $x_1$-axis.}
	
	\medskip 
	
	\noindent This time, we obtain an upper bound for $\rd_1u_1$ at the origin. We obtain a simple upper bound on the integrals $I$ by replacing $\omega(y)$ with $\nrm{\omega_0}_{L^\infty}$ in the region $\omega(y)>0$ ($\omega(y)<0$, resp.).  We obtain that \begin{equation} \label{eq:d1u1-upper}
	\begin{split}
	\rd_1 u_1(t,0,x_2) \le \frac{2}{\pi}  \nrm{\omega_0}_{L^\infty}\left( ( 1  + \eps_n) \ln \frac{1}{\bell} \right). 
	\end{split}
	\end{equation}
	
	\medskip
	
	\noindent \textit{Step 3. Bounds on $\rd_1u_1$ in the small ball region.}
	
	\medskip 
	
	\noindent In order to estimate $\partial_1u_1$  not only on the axis but also inside the small ball region, we shall use the classical estimates for the 2D Euler solutions: \begin{equation*}
	\begin{split}
	1 + \log\left( 1  + \frac{ L\nrm{\omega(t)}_{\dot{C}^{1}} }{\nrm{\omega_0}_{L^\infty}} \right) \le  \left( 1 + \log\left( 1  +  \frac{L \nrm{\omega_0}_{\dot{C}^{1}} }{\nrm{\omega_0}_{L^\infty}}  \right) \right) \exp(C\nrm{\omega_0}_{L^\infty}t) 
	\end{split}
	\end{equation*} (cf. \cite[Theorem 2.1]{KS}).  Since $\nrm{\omega_0}_{\dot{C}^1} \lesssim (\kpp\ell)^{-1}\nrm{\omega_0}_{L^\infty},$ we obtain that \begin{equation*} 
	\begin{split}
	& \log(1 +  \frac{ L\nrm{\omega(t)}_{\dot{C}^{1}} }{\nrm{\omega_0}_{L^\infty}} ) \le \log(L (\kpp\ell)^{-1}  ) e^{C\dlt} \le \log( L(\kpp\ell)^{-1} )^{1 + C\dlt}
	\end{split}
	\end{equation*} for $\kpp\ell \ll 1$ and $\dlt > 0$ small. Hence \begin{equation}\label{eq:omega-C1}
	\begin{split}
	\nrm{\omega(t)}_{\dot{C}^1} \le \frac{C}{L} \left( \frac{L}{\kpp \ell} \right)^{1+C\dlt } \nrm{\omega_0}_{L^\infty}. 
	\end{split}
	\end{equation} Then we use the singular integral estimate \begin{equation*} 
	\begin{split}
	& \nrm{\nb u(t)}_{\dot{C}^{\frac{1}{2}}} \lesssim \nrm{\omega(t)}_{\dot{C}^{\frac{1}{2}}} \lesssim L^{\frac{1}{2}} \nrm{\omega(t)}_{\dot{C}^1}^{\frac{1}{2}} \nrm{\omega_0}_{L^\infty}^{\frac{1}{2}} \lesssim L^{-\frac{1}{2}} \left( \frac{L}{\kpp \ell} \right)^{\frac{1}{2}+C\dlt } \nrm{\omega_0}_{L^\infty}. 
	\end{split}
	\end{equation*}  We then obtain for $x = (x_1,x_2) \in \calD$ (recall the definition of $\calD$ from \eqref{eq:needle}), \begin{equation}\label{eq:velgrad-Holder}
	\begin{split}
	|\nabla u(t,x) - \nabla u(t,0)| \lesssim \left( \frac{\tell}{L}\right)^{\frac{1}{2}}  \left( \frac{L}{\kpp \ell} \right)^{\frac{1}{2}+C\dlt } \nrm{\omega_0}_{L^\infty}. 
	\end{split}
	\end{equation} Therefore, we conclude that as long as $\tell$ is chosen in a way that \begin{equation}\label{eq:tell} 
	\begin{split}
	& \tell \lesssim (\kpp\ell)^{1+c\dlt}L^{-c\dlt},
	\end{split}
	\end{equation}  the  same lower and upper bounds for $\rd_1 u_1(t)$ given in \eqref{eq:d1u1-lower} and \eqref{eq:d1u1-upper} holds for $x = (x_1,x_2)$ (possibly with larger $\eps_n>0$).

	\medskip
	
	\noindent \textit{Step 4. Bounds on $\rd_1u_2$ and $\rd_2u_1$ in the small ball region.}
	
	\medskip 
	
	\noindent Along the axis $x_1 = 0$, we have vanishing of $u_1(t)$ from the odd symmetry for all $t$. In particular, taking a $x_2$-derivative, we also have that $\rd_2 u_1(t,0,x_2) = \rd_1 u_2(t,0,x_2) = 0$ for all $x_2$. Applying \eqref{eq:velgrad-Holder} under the condition \eqref{eq:tell} ensures that for $x \in \calD$, \begin{equation*} 
	\begin{split}
	& |\rd_2 u_1(t,x)| + |\rd_1 u_2(t,x)| \le C \nrm{\omega_0}_{L^\infty}.  
	\end{split}
	\end{equation*}
	
	\noindent The proof is now complete. 
\end{proof}

\begin{lemma}\label{lem:velgrad-global}
	Under the same assumptions in Lemma \ref{lem:nabla-u-local}, we have \begin{equation}\label{eq:nabla-u-global}
	\begin{split}
	\nrm{\rd_1u^\calL_{ {{n,1}}}(t)}_{L^\infty} 
&\le \frac{2}{\pi}\nrm{\omega_{n,0}^\calL}_{L^\infty} \left( (1 + \eps_n) \ln \frac{1}{\bell} + C \ln\left((1+\kpp) \left(\frac{L}{\ell}\right)^{c\dlt} \right) \right)\\
&
\lesssim
\|\omega_{n,0}^{\mathcal{L}}\|_{L^\infty}(1+\epsilon_n)\ln\frac{1}{\bar \ell}. 
	\end{split}
	\end{equation} Moreover, \begin{equation}\label{eq:nabla-u-global2}
	\begin{split}
	\nrm{\rd_2u^\calL_{ {{n,1}}}(t)}_{L^\infty}  + \nrm{\rd_1u^\calL_{ {{n,2}}}(t)}_{L^\infty}  \le C \dlt \nrm{\omega_{n,0}^\calL}_{L^\infty}  \ln \frac{1}{\bell}. 
	\end{split}
	\end{equation}
\end{lemma}
\begin{proof}
	We first prove \eqref{eq:nabla-u-global}. From the explicit formula \begin{equation*}
	\begin{split}
	\rd_1 u_1(t,x) = \frac{1}{\pi} P.V. \int_{\mathbb{R}^2} \frac{(y_1-x_1)(y_2-x_2)}{|y-x|^4} \omega(t,y) dy, 
	\end{split}
	\end{equation*}  we divide the integral into three regions: (i) $|x-y| < \nrm{\omega_0}_{L^\infty} \nrm{ \omega(t) }^{-1}_{\dot{C}^1}$, (ii)  $\frac{L}{2} \ge |x-y| \ge \nrm{\omega_0}_{L^\infty} \nrm{ \omega(t) }^{-1}_{\dot{C}^1}$, (iii) $|x-y| > \frac{L}{2}$. In (iii), the integral can be estimated by $C\nrm{\omega_0}_{L^\infty}$, and one estimates the integral in (ii) as in the proof of Lemma \ref{lem:nabla-u-local}, which gives the expression in \eqref{eq:nabla-u-global}, recalling the bound \begin{equation*} 
	\begin{split}
	& \frac{L\nrm{ \omega(t) }_{\dot{C}^1} }{\nrm{\omega_0}_{L^\infty}} \le C(\kpp\bell)^{-(1+c\dlt)}. 
	\end{split}
	\end{equation*}  Lastly, in the region (i), we write \begin{equation*} 
	\begin{split}
	\left| \int_{|x-y| < \frac{\nrm{\omega_0}_{L^\infty} }{\nrm{ \omega(t) }_{\dot{C}^1} }} \frac{(y_1-x_1)(y_2-x_2)}{|y-x|^4}( \omega(t,y) - \omega(t,x) ) dy \right| & \lesssim \nrm{ \omega(t) }_{\dot{C}^1} \int_{|x-y| < \frac{\nrm{\omega_0}_{L^\infty} }{\nrm{ \omega(t) }_{\dot{C}^1} }}  \frac{1}{|x-y|} dy \lesssim C\nrm{\omega_0}_{L^\infty}.
	\end{split}
	\end{equation*} This concludes the proof. 
	
	Turning to \eqref{eq:nabla-u-global2}, it suffices to show the estimate for $\rd_2u_1$ only. We do this again by estimating the explicit form of the singular integral kernel. However, using the fact that $\rd_2u_1$ is uniformly bounded when $\omega$ is given exactly by the Bahouri-Chemin stationary solution (this can be shown using either Fourier series with Poisson summation formula or radial-angular decomposition; cf. \cite{EJSVP,D,Den}), we just need to estimate the part where $\omega_n^\calL(t)$ is different from the Bahouri-Chemin solution. Moreover, without loss of generality we take $x = (x_1,x_2)$ with $0 \le x_1 \le x_2 \le \frac{L}{2}$ and we need to show a bound on the following: \begin{equation*} 
	\begin{split}
	& \int_{ \{y : |\omega_n^\calL(t,y)| \ne \nrm{\omega_n^\calL}_{L^\infty} \} }   \frac{(y_1-x_1)^2 - (y_2-x_2)^2}{|y-x|^4} \omega_n^\calL(t,y) dy.
	\end{split}
	\end{equation*} From our assumption that $x$ lies in the first quadrant, the main term in the integral comes from the strips $S_1 = \{ 0 \le x_1 \le L, 0 \le x_2 \le L \bell^{1-c\dlt} \}$ and $S_2 = \{  0 \le x_2 \le L, 0 \le x_1 \le L \bell^{1-c\dlt} \}$. We shall further assume that $x$ belongs to $S_1$ since otherwise then the kernel becomes less singular (and a similar argument gives the same bound). Then, we estimate \begin{equation*} 
	\begin{split}
	&  \int_{S_1}  \frac{(y_1-x_1)^2 - (y_2-x_2)^2}{|y-x|^4} \omega_n^\calL(t,y) dy  \\
	&\quad = \left[ \int_{|x-y| \le L \bell^{1+c\dlt} } + \int_{ L \bell^{1+c\dlt} < |x-y| \le L \bell^{1-c\dlt} } + \int_{S_1\backslash \{ y: |x-y| \le L\bell^{1-c\dlt} \}  } \right]  \frac{(y_1-x_1)^2 - (y_2-x_2)^2}{|y-x|^4} \omega_n^\calL(t,y) dy\\
	&\quad  =:  I + II + III 
	\end{split}
	\end{equation*} and then it is straightforward to bound terms $I$ and $II$: \begin{equation*} 
	\begin{split}
	|I| \lesssim  \nrm{\omega_0}_{L^\infty}
	\end{split}
	\end{equation*} (proceeding as in region (i) from the proof of \eqref{eq:nabla-u-global}) \begin{equation*} 
	\begin{split}
	& |II| \lesssim \nrm{\omega_0}_{L^\infty} \int_{ L \bell^{1+c\dlt} < |x-y| \le L \bell^{1-c\dlt}  } \frac{1}{|y-x|^2} dy \lesssim \dlt \ln\frac{1}{\bell}. 
	\end{split}
	\end{equation*} Finally, to estimate $III$ it suffices to bound the following ``rectangular'' integral (note that $y_2\ge y_1$ in this region): \begin{equation*} 
	\begin{split}
	&  \int_{ [0,L \bell^{1-c\dlt}] \times [L\bell^{1-c\dlt},L] } \frac{y_2^2-y_1^2}{|y|^4} dy  = \int_{ [0,1]\times [1,\bell^{-(1-c\dlt)}] } \frac{z_2^2-z_1^2}{|z|^4} dz \\
	& \quad = \int_{1}^{\bell^{-(1-c\dlt)}} \frac{1}{1+z_2^2} dz_2 < + \infty. 
	\end{split}
	\end{equation*} Indeed, this type of rectangular integral bound has appeared already in \cite{KS,Z}. The proof is complete. 
\end{proof}

\subsubsection{Estimates for trajectories and Lagrangian deformation}

We keep working in the time interval $[0,\dlt\nrm{\omega_0}_{L^\infty}^{-1}]$ and we shall first extract a smaller ball region $\calD'_{\dlt}$ such that $\eta(t,\calD'_{\dlt}) \subset \calD$ during this time interval. We then prove estimates regarding the Lagrangian deformation $\nabla\eta(t,x)$ for $x \in \calD'_{\dlt}$. 

First, it is not difficult to show that $u_2(t,x_1,x_2) < 0$ when 
$|x| \le \tell$. (For a proof, one can see the Key Lemma from \cite{KS} and \cite{Z}. This piece of information will not be essential in our arguments.) Next, we use that (assuming $\eta_1(0) > 0$) \begin{equation*} 
\begin{split}
& \dot{\eta}_1(t) = u_1(t,\eta(t)) \le \nrm{\rd_1 u_1}_{L^\infty(\calD)} \eta_1(t)
\end{split}
\end{equation*} which is valid as long as $\eta(t) \in \calD$. We have used that $u_1(t,0,\eta_2) = 0$ holds in the above estimate. Assuming formally that $\eta(t) \in \calD$, we have from \eqref{eq:d1u1-upper} that \begin{equation*} 
\begin{split}
& \eta_1(t) \le x_1 \exp\left( ct  \nrm{\omega_0}_{L^\infty} (C + \ln \frac{(1+\kpp)A}{\bell})  \right) \le x_1(1 + \dlt)\left( \frac{(1+\kpp)A}{\bell} \right)^{c\delta} < 2 x_1 \bell^{-(1+c\dlt)c\dlt} < x_1 \bell^{-c\dlt}
\end{split}
\end{equation*} (by taking $\dlt > 0$ sufficiently small; recall that the value of $c$ can change several times even within a single line). Hence we may define the region \begin{equation}\label{eq:needle-small} 
\begin{split}
& \calD'_\dlt = \{ |x| < \tell \cdot  \bell^{c\dlt}\} 
\end{split}
\end{equation} so that $\eta(t,\calD'_{\dlt}) \subset \calD$ for $t \in [0,\dlt\nrm{\omega_0}_{L^\infty}^{-1}]$. In the remainder of this section, we always take $x \in \calD'_{\dlt}$ and $t \in [0,\dlt\nrm{\omega_0}_{L^\infty}^{-1}]$.
 
\begin{lemma}[Creation of large Lagrangian deformation]\label{lem:LLD}
	Let us denote $\eta(t) = \eta_n(t)$ to be the flow associated with $\omega_n(t)$ from Lemma \ref{lem:nabla-u-local}. For $x \in \calD'_{\dlt}$ and $t \in [0,\dlt\nrm{\omega_0}_{L^\infty}^{-1}]$, we have that \begin{equation}\label{eq:LLD-1}
	\begin{split}
	\exp\left( \frac{2}{\pi}\nrm{\omg^\calL_{n,0}}_{L^\infty} t (1-\eps_n-C\dlt) \ln \frac{1}{\bell}  \right) < \rd_1 \eta_1(t,x) < \exp\left( \frac{2}{\pi}\nrm{\omg^\calL_{n,0}}_{L^\infty} t (1+\eps_n+C\dlt) \ln \frac{1}{\bell}  \right)
	\end{split}
	\end{equation} and   \begin{equation}\label{eq:flow-upperbound}
	\begin{split}
	|\rd_1\eta_2(t,x)| + |\rd_2\eta_1(t,x)| \le \eps_n \rd_1\eta_1(t,x)
	\end{split}
	\end{equation} where $\eps_n\rightarrow 0$ as $n\rightarrow+\infty$.
\end{lemma}

\begin{proof}

	Now we consider the following system of ODEs: for each $x $,  denoting for simplicity $\eta := \eta(t,x)$ and $\rd_i\eta_j(t):= \rd_i\eta_j(t,x)$, \begin{equation}\label{eq:stretch-ODE}
	\begin{split}
	\frac{d}{dt} \rd_1 \eta_1(t) &= \rd_1 u_1(t,\eta ) \rd_1\eta_1(t) + \rd_2 u_1(t,\eta ) \rd_1\eta_2(t) \\
	\frac{d}{dt} \rd_1 \eta_2(t) &= -\rd_1 u_1(t,\eta ) \rd_1\eta_2(t) + (\rd_2 u_1(t,\eta ) +\omega(t,\eta))  \rd_1\eta_1(t) .
	\end{split}
	\end{equation} As long as $\eta \in \calD$ we have that $\omega(t,\eta) = 0$. We shall prove that for each fixed $x$, we have both \begin{equation}\label{eq:positive} 
	\begin{split}
	& \rd_1\eta_1(t) >0, \quad \eps \rd_1\eta_1(t) - |\rd_1\eta_2(t)| > 0. 
	\end{split}
	\end{equation} Here $\eps := \inf_{[0,\delta/\nrm{\omega_0}_{L^\infty}] \times \calD } \frac{|\rd_2u_1|}{\rd_1u_1} >0$. Note that both inequalities are satisfied for some nonempty interval of time containing $t = 0$, since $\rd_1\eta_1(t=0) = 1$ and $\rd_1\eta_2(t = 0) = 0$. Multiplying the first equation of \eqref{eq:stretch-ODE} by $\eps$ and subtracting the second, \begin{equation*} 
	\begin{split}
	&\frac{d}{dt} \left( \eps \rd_1\eta_1 - |\rd_1\eta_2| \right)  \ge \rd_1u_1(\eps\rd_1\eta_1 + |\rd_1\eta_2|) - \eps |\rd_2u_1|\rd_1\eta_2| - |\rd_2u_1| |\rd_1\eta_1| \\
	&\quad\ge  \rd_1u_1(\eps\rd_1\eta_1 -|\rd_1\eta_2|) + 2 \rd_1u_1|\rd_1\eta_2| - \eps |\rd_2u_1|\rd_1\eta_2| - |\rd_2u_1| |\rd_1\eta_1| \\
	&\quad\ge (\rd_1u_1- \frac{1}{\eps|\rd_2u_1|})(\eps\rd_1\eta_1 -|\rd_1\eta_2|) + \frac{1}{\eps}|\rd_2u_1|(\eps \rd_1\eta_1 - |\rd_1\eta_2|)+ 2 \rd_1u_1|\rd_1\eta_2| - \eps |\rd_2u_1|\rd_1\eta_2| - |\rd_2u_1| |\rd_1\eta_1|  \\
	\end{split}
	\end{equation*} and \textit{assuming} $\rd_1\eta_1 > 0$, \begin{equation*} 
	\begin{split}
	&\frac{1}{\eps}|\rd_2u_1|(\eps \rd_1\eta_1 - |\rd_1\eta_2|)+ 2 \rd_1u_1|\rd_1\eta_2| - \eps |\rd_2u_1|\rd_1\eta_2| - |\rd_2u_1| |\rd_1\eta_1| \\
	&\quad = (2\rd_1u_1 - \frac{1}{\eps}|\rd_2u_1| - \eps |\rd_2u_1|)|\rd_1\eta_2| > 0. 
	\end{split}
	\end{equation*} Hence this shows that under the assumption  $\rd_1\eta_1 > 0$, we can propagate in time that $\eps \rd_1\eta_1 - |\rd_1\eta_2|$. Of course the latter again implies $\rd_1\eta_1>0$. Therefore a simple continuity argument establishes \eqref{eq:positive}. 
	
	Returning to \eqref{eq:stretch-ODE}, we have that \begin{equation*} 
	\begin{split}
	&(1 - \eps^2) \rd_1 u_1(t,\eta ) \rd_1\eta_1(t) < \frac{d}{dt} \rd_1 \eta_1(t) < (1 + \eps^2) \rd_1 u_1(t,\eta ) \rd_1\eta_1(t)
	\end{split}
	\end{equation*} and integrating in time gives, with $\eps = \eps_n \rightarrow 0$ as $n\rightarrow+\infty$, \begin{equation*} 
	\begin{split}
	& \exp\left( \frac{2}{\pi}\nrm{\omg^\calL}_{L^\infty} t (1-\eps_n-C\dlt) \ln \frac{1}{\bell}  \right) < \rd_1 \eta_1(t) < \exp\left( \frac{2}{\pi}\nrm{\omg^\calL}_{L^\infty} t (1+\eps_n+C\dlt) \ln \frac{1}{\bell}  \right)
	\end{split}
	\end{equation*} This finishes the proof. 
\end{proof}

\subsubsection{Estimates for the gradient of the vorticity}

In this section, we shall establish that for $p = 2, +\infty$, we have the following sharp estimate on $\nb\omg_n^\calL$: \begin{equation}\label{eq:nb-omega-Lp}
\begin{split}
\nrm{\nb\omega_n^\calL(t)}_{L^p} \le  \nrm{\nb\omega_{n,0}^\calL}_{L^p}\exp\left( \dlt(1+C\dlt) \frac{\nrm{\rd_1 u^\calL_{n,1} }_{L^\infty_{t,x}}}{\nrm{\omega_{n,0}^\calL}_{L^\infty}} \right)
\end{split}
\end{equation} 
for $t\in [0,\delta/\|\omega_{n,0}^{\mathcal L}\|_{L^\infty}]$.
The same estimate holds uniformly for $\nb\omg_n^{\calL,\nu}$ with any $\nu > 0$ (possibly with some different constant $C>0$). Here $\omega_n^{\calL,\nu}$ is defined by the solution of 2D Navier-Stokes \begin{equation*} 
\begin{split}
& \rd_t \omega_n^{\calL,\nu} + u_n^{\calL,\nu} \cdot \nb \omega_n^{\calL,\nu} = \nu\lap \omega_n^{\calL,\nu}, \\
& \nb \cdot u_n^{\calL,\nu} = 0, \\
& \omega_n^{\calL,\nu}(t = 0) = \omega_{n,0}^{\calL}. 
\end{split}
\end{equation*} To see that \eqref{eq:nb-omega-Lp} holds, simply take the gradient of the equation for $\omega_n^\calL$: \begin{equation*} 
\begin{split}
& \rd_t\nb \omega_n^{\calL} + u_n^{\calL} \cdot \nb (\nb\omg_n^\calL) = ( \nb u_n^\calL )^T \nb \omg_n^\calL. 
\end{split}
\end{equation*} Taking the dot product with $\nb \omega_n^{\calL}$ and integrating in space gives \begin{equation*} 
\begin{split}
& \frac{1}{2} \frac{d}{dt} \nrm{\nb \omega_n^{\calL}}_{L^2}^2 \le \left| \int  \nb \omega_n^{\calL} \cdot ( \nb u_n^\calL )^T \nb \omg_n^\calL \right|. 
\end{split}
\end{equation*} Recall from \eqref{eq:nabla-u-global}--\eqref{eq:nabla-u-global2} that the $2\times 2$ matrix $\nb u_n^\calL$ has the following structure: \begin{equation*} 
\begin{split}
& M = \begin{pmatrix}
X & O(X\dlt) \\
O(X\dlt) & -X  
\end{pmatrix}
\end{split}
\end{equation*} where $X \gg 1$. Eigenvalues of $M$, in absolute value, has size $X(1 \pm O(\dlt))$. In particular we see that for any $2\times 1$ vector $v$, \begin{equation*} 
\begin{split}
& |v^T Mv| \le X(1 \pm O(\dlt))|v|^2 . 
\end{split}
\end{equation*} Applying this observation with $X = \nrm{\rd_1 u^\calL_{n,1} }_{L^\infty_{t,x}}$ gives \eqref{eq:nb-omega-Lp} for $p = 2$, since we have the bound \begin{equation*} 
\begin{split}
& \frac{1}{2} \frac{d}{dt} \nrm{\nb \omega_n^{\calL}}_{L^2}^2 \le (1 + C\dlt) \nrm{\rd_1 u^\calL_{n,1} }_{L^\infty_{t,x}} \nrm{\nb \omega_n^{\calL}}_{L^2}^2  .
\end{split}
\end{equation*} The argument for $p = +\infty$ is similar. (Indeed the same estimate holds uniformly for $p$ in $1\le p \le +\infty$.) We shall use this type of argument  several times in the following. 

Based on \eqref{eq:nb-omega-Lp}, let us obtain a sharp bound for the second gradient $\nrm{\nabla^2 u^\calL_n(t)}_{L^\infty_{\calD}}$. Note that each component of $\nabla^2 u^\calL_n(t)$ is a singular integral transform applied to a derivative of $\omega_n^\calL$, which vanishes both near and away from the axes. Proceeding similarly as in the proof of Lemma \ref{lem:velgrad-global}, we estimate for $x \in \calD$ \begin{equation*} 
\begin{split}
& |\nabla^2 u^\calL_n(t,x)| \le C \nrm{\nabla\omg^\calL_n(t)}_{L^\infty} \left( 1 + \int_{ \ell \bell^{ct \nrm{\omg_{n,0}^\calL}_{L^\infty} }   }^{\ell \bell^{-ct \nrm{\omg_{n,0}^\calL}_{L^\infty} } } \frac{dr}{r}\right)  \le C(1 + t \ln \frac{1}{\bell}) \nrm{\omg_{n,0}^\calL}_{L^\infty}\nrm{\nb\omg_n^\calL(t) }_{L^\infty} . 
\end{split}
\end{equation*} Hence, as long as $0 \le t \le \dlt/\nrm{\omega^\calL_{n,0}}_{L^\infty}$, \begin{equation}\label{eq:velsecondgrad}
\begin{split}
\nrm{\nabla^2 u^\calL_n(t)}_{L^\infty_{\calD}} \le C(1 + t\nrm{\omg_{n,0}^\calL}_{L^\infty} \ln \frac{1}{\bell}) \nrm{\nb\omg_{n,0}^\calL  }_{L^\infty} \exp\left( t(1+C\dlt) \nrm{\rd_1 u^\calL_{n,1}}_{L^\infty_{t,x}} \right). 
\end{split}
\end{equation}

\subsection{Setup for the small-scale vorticity}\label{subsec:small} With a length scale $\tell = \tell_n \ll \ell$ and small $\dlt>0$, we recall the definition of $\calD'_\dlt$ from \eqref{eq:needle-small}. Define $u^\calS_{n,0} \in C^\infty(\bbT^3)$ in a way that \begin{equation*} 
\begin{split}
& u^\calS_{n,0}(x_1,x_2,x_3) = \begin{pmatrix}
0 \\
0 \\
M x_2 
\end{pmatrix}
\end{split} 
\end{equation*} on $\calD'_\dlt \times \bbT$ and $u^\calS_{n,0} \equiv 0$ on $(\bbT^2\backslash\calD)\times\bbT $. We may arrange in addition that $u^\calS_{n,0}$ is only a function of $x_2$ and has vanishing first and second components. Therefore we shall identify $u^\calS_{n,0}$ with its third component with some abuse of notation. Note that $u^\calS_{n,0}$  is divergence-free. Note that taking the curl gives \begin{equation*} 
\begin{split}
& \omg^\calS_{n,0} := \nabla \times u^\calS_{n,0} = \begin{pmatrix}
M \\
0 \\
0 
\end{pmatrix} \quad \mbox{on} \quad \calD'_\dlt \times \bbT 
\end{split}
\end{equation*} and we see that $M \le \nrm{\omega^\calS_{n,0}}_{L^\infty} \le 2M$ (by redefining $u^\calS_{n,0}$ outside $\calD'_\dlt$ if necessary). 
\begin{remark}
We simply have $\|u_{n,0}^{\mathcal S}\|_{L^2}^2\approx M\int_0^{\tilde\ell}\int_0^{\tilde\ell}x_2^2dx_2dx_3
\approx M\tilde\ell^4\approx\tilde\ell^2\|\omega_{n,0}^S\|_{L^2}^2$.
\end{remark}

\subsubsection{Estimates for the small-scale vorticity} We consider the equation \begin{equation*} 
\begin{split}
& \rd_t u_n^\calS + u_n^\calL \cdot \nb u_n^\calS = 0. 
\end{split}
\end{equation*} Since $u_n^\calL$ is divergence-free, we immediately have \begin{equation*} 
\begin{split}
& \nrm{u_n^\calS}_{L^p} =  \nrm{u_{n,0}^\calS}_{L^p} 
\end{split}
\end{equation*} for all $1 \le p \le + \infty$. Next, taking the curl gives \begin{equation}\label{eq:omg-small}
\begin{split}
\rd_t \omega_n^\calS + (u_n^\calL\cdot\nabla)\omega_n^\calS = \nabla u_n^\calL \omega_n^\calS,
\end{split}
\end{equation} and we obtain that \begin{equation*} 
\begin{split}
& \frac{1}{2} \frac{d}{dt} \nrm{\omg_n^\calS}_{L^2}^2 \le \left| \int \omega_n^\calS \cdot \nb u^\calL_n \cdot \omega_n^\calS \right| \le  (1 + C\dlt)  \nrm{\rd_1u^\calL_{n,1}}_{L^\infty_{\calD}}  \nrm{\omg_n^\calS}_{L^2}^2, 
\end{split}
\end{equation*} where we have used that $$ \nrm{\rd_2u^\calL_{n,1}}_{L^\infty_{\calD}} + \nrm{\rd_1u^\calL_{n,2}}_{L^\infty_{\calD}} \lesssim \dlt  \nrm{\rd_1u^\calL_{n,1}}_{L^\infty_{\calD}}.$$ We shall use this observation frequently in the following. Similarly, it is not difficult to see (repeating a bootstrap argument as in the proof of Lemma \ref{lem:LLD}) that \begin{equation*} 
\begin{split}
& \frac{d}{dt} \nrm{\omg_n^\calS}_{L^\infty} \le (1 + C\dlt)  \nrm{\rd_1u^\calL_{n,1}}_{L^\infty_{\calD}}  \nrm{\omg_n^\calS}_{L^\infty}. 
\end{split}
\end{equation*} Integrating in time, we see that for $p = 2, +\infty$ (actually this holds uniformly for any $1\le p \le +\infty$) \begin{equation}\label{eq:small-omega-Lp}
\begin{split}
\nrm{\omega_n^\calS(t)}_{L^p} \le \nrm{\omega_{n,0}^\calS}_{L^p} e^{(1+C\dlt)\int_0^t\nrm{ \rd_1 u_{n,1}^\calL(\tau)}_{L^\infty_\calD}d\tau} \le \nrm{\omega_{n,0}^\calS}_{L^p}\exp\left( \dlt(1+C\dlt) \frac{\nrm{\rd_1 u^\calL_{n,1} }_{L^\infty_{t,x}}}{\nrm{\omega_{n,0}^\calL}_{L^\infty}} \right). 
\end{split}
\end{equation} It is not difficult to see that \begin{equation}\label{eq:omg-small-comp} 
\begin{split}
& \omega^\calS_{n,1} \le C\dlt |\omega^\calS_{n,2}|
\end{split}
\end{equation} pointwise in space and time. This can be seen directly from \eqref{eq:omg-small} but it is easy to obtain from Lemma \ref{lem:LLD} and the following Cauchy formula \begin{equation*} 
\begin{split}
& \omega_n^\calS(t,\eta(t,x)) = D\eta(t,x) \omega^\calS_{n,0}(x)
\end{split}
\end{equation*} since $\omega_{n,1}^\calS(t = 0) = 0$.

Similarly, from \begin{equation*}
\begin{split}
\rd_t \omega_n^{\calS,\nu} + (u_n^{\calL,\nu}\cdot\nabla)\omega_n^{\calS,\nu} = \nabla u_n^{\calL,\nu} \omega_n^{\calS,\nu} + \nu \lap \omega_n^{\calS,\nu},
\end{split}
\end{equation*} one can obtain that the estimates \eqref{eq:small-omega-Lp} are valid for $\omega_n^{\calS,\nu}$ uniformly for any $\nu>0$, perhaps with some different absolute constant $C>0$.

We shall need just one more estimate: take the first component of the equation for $\omega_n^\calS$ and differentiating gives \begin{equation*} 
\begin{split}
& D_t \rd_1 \omg^\calS_{n,1} = - \rd_1 u_{n,2}^\calL \rd_2\omega_{n,1}^\calS + \rd_2u_{n,1}^\calL \rd_1\omega_{n,2}^\calS + \rd_1\rd_2u_{n,1}^\calL \omg_{n,2}^\calS + \rd_1 \rd_2 u_{n,1}^\calL \omg_{n,1}^\calS 
\end{split}
\end{equation*} and \begin{equation*} 
\begin{split}
& D_t \rd_2 \omg^\calS_{n,1} = 2 \rd_1u_{n,1}^\calL \rd_2\omg_{n,1}^\calS - \rd_2u_{n,1}^\calL\rd_1\omg_{n,1}^\calS + \rd_2 u_{n,1}^\calL \rd_2\omg_{ {n,2}}^\calS + \rd_2\rd_2 u_{n,1}^\calL \omg_{n,2}^\calS + \rd_2\rd_1 u_{n,1}^\calL \omg_{n,1}^\calS 
\end{split}
\end{equation*} where we have written \begin{equation*} 
\begin{split}
&D_t = \rd_t + u^\calL_n\cdot\nb 
\end{split}
\end{equation*} for simplicity. It is not difficult to see that we have \begin{equation*} 
\begin{split}
& \nrm{\rd_1\omega_{n,1}^\calS(t)}_{L^\infty} + \nrm{\rd_2\omega_{n,2}^\calS(t)}_{L^\infty} \le C \nrm{\rd_2\omega_{n,1}^\calS(t)}_{L^\infty} . 
\end{split}
\end{equation*} We then estimate using \eqref{eq:omg-small-comp} that \begin{equation*} 
\begin{split}
& \frac{d}{dt} \nrm{\rd_2 \omg_{n,1}^\calS(t)}_{L^\infty} \le 2(1+C\dlt)\nrm{ \rd_1u^\calL_{n,1} }_{L^\infty_{t,x}}  \nrm{\rd_2 \omg_{n,1}^\calS(t)}_{L^\infty} + (1+C\dlt) \nrm{\nabla^2 u^\calL_n(t)}_{L^\infty} \nrm{\omg^\calS_{n,1}(t)}_{L^\infty} .
\end{split}
\end{equation*} Using Gronwall's inequality together with \eqref{eq:velsecondgrad}, we obtain from  \begin{equation}
\begin{split}
 &\nrm{\rd_2 \omg_{n,1}^\calS(t)}_{L^\infty} \le  \nrm{\nb\omg_{n,0}^\calS}_{L^\infty} \exp\left( 2t(1+C\dlt)\nrm{ \rd_1u^\calL_{n,1} }_{L^\infty_{t,x}}  \right) \\
 &\quad  +  (1+C\dlt)  \int_0^t \exp\left( 2(t-s)(1+C\dlt)\nrm{ \rd_1u^\calL_{n,1} }_{L^\infty_{t,x}}  \right)  \nrm{\nabla^2 u^\calL_n(s)}_{L^\infty} \nrm{\omg^\calS_{n,1}(s)}_{L^\infty}  ds
\end{split}
\end{equation} that \begin{equation}\label{eq:omg-grad-small-1}
\begin{split}
\nrm{\rd_2 \omg_{n,1}^\calS(t)}_{L^\infty} \le   \left(\nrm{\nb\omg_{n,0}^\calS}_{L^\infty}  + t(1 + \dlt \ln \frac{1}{\bell})\nrm{\nb\omg^\calL_{n,0}}_{L^\infty} \nrm{\omg^\calS_{n,0}}_{L^\infty} \right) \exp\left( 2t(1+C\dlt)\nrm{ \rd_1u^\calL_{n,1} }_{L^\infty_{t,x}}  \right). 
\end{split}
\end{equation} One can similarly estimate $\rd_1\omg_{n,1}^\calS$ and the gradient of the second component in a parallel manner; it turns out that $\nrm{\nabla \omg^\calS_{n,2}(t)}_{L^\infty}$ satisfies the estimate \eqref{eq:omg-grad-small-1} as well. We omit the details. Moreover, in $L^2$ we can obtain a corresponding estimate: \begin{equation}\label{eq:omg-grad-L2}
\begin{split}
\nrm{\nb\omg_{n}^\calS(t)}_{L^2} \le   \left(\nrm{\nb\omg_{n,0}^\calS}_{L^2}  + t(1 + \dlt \ln \frac{1}{\bell})\nrm{\nb\omg^\calL_{n,0}}_{L^\infty} \nrm{\omg^\calS_{n,0}}_{L^2} \right) \exp\left( 2t(1+C\dlt)\nrm{ \rd_1u^\calL_{n,1} }_{L^\infty_{t,x}}  \right). 
\end{split}
\end{equation}


\subsection{Inviscid limits}\label{subsec:inviscid}

As before, we shall always take $t \in [0,\dlt/\nrm{\omega_0}_{L^\infty}]$ throughout this section. We obtain sharp upper bounds for the $L^2$ and $\dot{H}^1$ differences between the Euler and Navier-Stokes velocities, both for the large and small scales. 

\subsubsection{$L^2$ for the large scale}\label{subsubsec:inviscid-1} We define \begin{equation*} 
\begin{split}
& I^\calL(t) := \nrm{u_n^{\calL,\nu} - u_n^\calL}_{L^2}^2. 
\end{split}
\end{equation*} 
We compare the 2D Euler and Navier-Stokes equations of the velocity:
\begin{equation*}
\begin{split}
&\rd_t u_n^{\calL,\nu} + u_n^{\calL,\nu} \cdot \nabla u_n^{\calL,\nu} + \nabla p_n^{\calL,\nu} = \nu \lap u_n^{\calL,\nu} , \\
&\rd_t u_n^\calL + u_n^\calL\cdot\nabla u_n^\calL + \nabla p_n^\calL = 0. 
\end{split}
\end{equation*} 
Then, we see that 
\begin{equation*}
\begin{split}
\frac{1}{2}\frac{d}{dt} \nrm{u_n^{\calL,\nu} - u_n^\calL}_{L^2}^2 + \int( u_n^\calL - u_n^{\calL,\nu}) \cdot \nabla u_n^\calL \cdot (u_n^{\calL,\nu} - u_n^\calL)  
= \nu \int \lap u_n^{\calL,\nu} \cdot (u_n^{\calL,\nu} - u_n^\calL) .
\end{split}
\end{equation*} We handle the right hand side as follows: \begin{equation*}
\begin{split}
\nu \int \lap u_n^{\calL,\nu} \cdot (u_n^{\calL,\nu} - u_n^\calL)  = -\nu\int |\nabla u_n^{\calL,\nu}|^2 + \nu \int \nabla u_n^{\calL,\nu} : \nabla u_n^\calL  \le C\nu \nrm{\nb u_n^\calL}_{L^2}^2 . 
\end{split}
\end{equation*} Moreover, inspecting the second term on the left hand side, we may bound \begin{equation*} 
\begin{split}
& \left| \int( u_n^\calL - u_n^{\calL,\nu}) \cdot \nabla u_n^\calL \cdot (u_n^{\calL,\nu} - u_n^\calL)   \right| \le \nrm{\rd_1u_{n,1}^\calL}_{L^\infty} I^\calL + C(\nrm{\rd_1u_{n,2}^\calL}_{L^\infty} + \nrm{\rd_2u_{n,1}^\calL}_{L^\infty} )I^\calL .
\end{split}
\end{equation*} Hence, appealing to the global bound \eqref{eq:nabla-u-global}--\eqref{eq:nabla-u-global2}, 
\begin{equation*}
\begin{split}
\frac{d}{dt}  I^\calL   \le  2(1+C\dlt)\nrm{\rd_1 u_{n,1}^\calL}_{L^\infty} I^\calL  + C\nu \nrm{\nabla u_n^\calL}_{L^2}^2.
\end{split}
\end{equation*} Using that $\nrm{\nabla u_n^\calL}_{L^2} \le C \nrm{\omega_n^\calL}_{L^2} \le C \nrm{\omega_{n,0}^\calL}_{L^2}$ and $I^\calL(0) = 0$,  we arrive at \begin{equation}\label{eq:large-L2} 
\begin{split}
& I^\calL(t) \le C\nu t \nrm{\omega^\calL_{n,0}}_{L^2}^2 \exp\left(2\dlt(1+C\dlt) \frac{\nrm{\rd_1 u_{n,1}^\calL}_{L^\infty_{t,x}}}{\nrm{\omega^\calL_0}_{L^\infty}} \right). 
\end{split}
\end{equation} The exponential term on the right hand side will appear frequently, so we shall introduce notation \begin{equation}\label{eq:calE} 
\begin{split}
& \calE := \exp\left( (1 + C\dlt) \frac{\nrm{\rd_1 u_{n,1}^\calL}_{L^\infty_{t,x}}}{\nrm{\omega^\calL_0}_{L^\infty}} \right).
\end{split}
\end{equation} Note that \begin{equation*}
\begin{split}
\calE \approx \bell^{-\frac{2}{\pi}(1+c\dlt)} \gg 1. 
\end{split}
\end{equation*}

\subsubsection{$L^2$ for the small scale}\label{subsubsec:inviscid-2}

Now we set \begin{equation*}
\begin{split}
I^\calS(t) := \nrm{u_n^{\calS,\nu} - u_n^{\calS}}_{L^2}^2 
\end{split}
\end{equation*} and again note that $I^\calS(0) = 0$. Compare the equations satisfied by $u_n^S$ and $u_n^{S,\nu}$: \begin{equation*}
\begin{split}
\rd_t u^\calS_{n} + u^\calL_n \cdot \nabla u^\calS_n = 0 ,
\end{split}
\end{equation*} \begin{equation*}
\begin{split}
\rd_t u^{\calS,\nu}_{n} + u^{\calL,\nu}_n \cdot \nabla u^{\calS,\nu}_n = \nu\lap u^{\calS,\nu}_n . 
\end{split}
\end{equation*} Proceeding similarly as in the above, we have \begin{equation*}
\begin{split}
\frac{1}{2} \frac{d}{dt} I^\calS  \le  \nrm{\omg_n^S}_{L^\infty} (I^\calL)^\frac{1}{2} (I^\calS)^{\frac{1}{2}}  + C\nu \nrm{\nabla u_n^S}_{L^2}^2. 
\end{split}
\end{equation*} It was crucially used that \begin{equation*} 
\begin{split}
& \left| \int \nb u_n^\calS \cdot ( u_n^\calL - u_n^{\calL,\nu}) (u_n^\calS - u_n^{\calS,\nu})  \right| \le \nrm{\omg^\calS_n}_{L^\infty} (I^\calS I^\calL )^\frac{1}{2};
\end{split}
\end{equation*} recall that $\omega^\calS_n$ is simply \begin{equation*} 
\begin{split}
& \omega^\calS_n = \begin{pmatrix}
\rd_1 u^\calS_n \\
-\rd_2 u^\calS_n \\
0 
\end{pmatrix}. 
\end{split}
\end{equation*} Using \eqref{eq:large-L2} we write for simplicity \begin{equation*} 
\begin{split}
&  \frac{d}{dt} I^\calS  \le  A (tI^\calS)^{\frac{1}{2}} + B
\end{split}
\end{equation*} where $A$ and $B$ are positive constants defined by \begin{equation*} 
\begin{split}
& A =  \nu^\frac{1}{2} \nrm{\omg^\calS_{n,0}}_{L^\infty}  \nrm{\omega_{n,0}^\calL}_{L^2}\calE^{2\dlt},\quad B = C\nu \nrm{\omg^\calS_{n,0}}_{L^2}^2\calE^{2\dlt}. 
\end{split}
\end{equation*} We have used \eqref{eq:small-omega-Lp}. To estimate $I^\calS$, we instead estimate the solution of the ODE \begin{equation*} 
\begin{split}
& \frac{d}{dt} X = A(tX)^{\frac{1}{2}} + B,\quad X(0) = 0. 
\end{split}
\end{equation*}  We estimate $X(t)$ differently in $0 \le t < t^*$ and $t>t^*$; here $t^*>0$ is the solution to \begin{equation*} 
\begin{split}
& A(t^*X(t^*))^{\frac{1}{2}} = B
\end{split}
\end{equation*} which is uniquely well-defined since initially $B> A(tX)^{\frac{1}{2}} $ and $A(tX)^{\frac{1}{2}} $ is strictly increasing in time. Then, we have trivially \begin{equation*} 
\begin{split}
&Bt^* \le X(t^*) \le 2B t^*
\end{split}
\end{equation*} so using the definition of $t^*$ above, we deduce \begin{equation*} 
\begin{split}
& t^* = c \frac{B^{\frac{1}{2}}}{A},
\end{split}
\end{equation*} with some absolute constant $\frac{1}{2} \le c \le 2$. In turn, this implies that \begin{equation*} 
\begin{split}
& X(t^*) \le 4 \frac{B^{\frac{3}{2}}}{A} 
\end{split}
\end{equation*} Next, for $t > t^*$, we have \begin{equation*} 
\begin{split}
&2 X^{\frac{1}{2}} \frac{d}{dt} X^{\frac{1}{2}} =  \frac{d}{dt}X \le 2A(tX)^{\frac{1}{2}} 
\end{split}
\end{equation*} and integrating in time gives \begin{equation*} 
\begin{split}
&  X^{\frac{1}{2}} (t) \le  X^{\frac{1}{2}}(t^*) + (t-t^*)^{\frac{3}{2}} A . 
\end{split}
\end{equation*} Using the above upper bound for $X(t^*)$ and squaring both sides, we conclude that \begin{equation*} 
\begin{split}
& X(t) \le C( \frac{B^{\frac{3}{2}}}{A} + t^3A^2  ). 
\end{split}
\end{equation*} Recalling the expressions for $A$ and $B$, we deduce that \begin{equation*} 
\begin{split}
& I^\calS(t) \le X(t) \le C  \left( \frac{\nu\nrm{\omg^\calS_{n,0}}_{L^2}^3}{\nrm{\omg^\calS_{n,0}}_{L^\infty} \nrm{\omega^\calL_{n,0}}_{L^2} } \calE^{-3\dlt}+ \nu t^3 \nrm{\omg^\calS_{n,0}}_{L^\infty}^2 \nrm{\omg^\calL_{n,0}}_{L^2}^2  \right)  \calE^{4\dlt} .          
\end{split}
\end{equation*} Simply replacing $t$ on the right hand side by $\dlt/\nrm{\omega^\calL_{n,0}}_{L^\infty}$, we arrive at \begin{equation}\label{eq:small-L2}
\begin{split}
I^\calS(t) \le C\nu \frac{\nrm{\omg^\calS_{n,0}}_{L^\infty}^2}{\nrm{\omega^\calL_{n,0}}_{L^2}} \left(   \frac{ \nrm{\omg^\calS_{n,0}}_{L^2} }{\nrm{\omg^\calS_{n,0}}_{L^\infty}} \calE^{-\dlt}  + \dlt   \frac{\nrm{\omg^\calL_{n,0}}_{L^2}}{\nrm{\omg^\calL_{n,0}}_{L^\infty}}    \right)^3 \calE^{4\dlt}. 
\end{split}
\end{equation}

\subsubsection{$\dot{H}^1$ for the large scale}\label{subsubsec:inviscid-3}  
We define \begin{equation*}
\begin{split}
II^{\calL} = \nrm{\omega_n^{\calL,\nu} - \omega_n^{\calL}}_{L^2}^2 . 
\end{split}
\end{equation*} From the equations 
\begin{equation*}
\begin{split}
&\rd_t\omega_n^{\calL,\nu} + u_n^{\calL,\nu}\cdot\nabla \omega_n^{\calL,\nu} = \nu\lap \omega_n^{\calL,\nu},  \\
&\rd_t\omega_n^\calL + u_n^\calL\cdot\nabla \omega_n^\calL = 0, 
\end{split}
\end{equation*} we obtain \begin{equation*}
\begin{split}
\frac{1}{2}\frac{d}{dt} II^{\calL} \le \nrm{\nabla \omega^{\calL}_{n}}_{L^\infty}\left(I^\calL II^\calL\right)^{\frac{1}{2}} + C\nu \nrm{\nabla\omega_n^\calL}_{L^2}^2. 
\end{split}
\end{equation*} We use \eqref{eq:large-L2} and \eqref{eq:nb-omega-Lp} to bound $I^\calL$ and $\nrm{\nabla \omega^{\calL}_{n}}_{L^p}$, respectively: 
\begin{equation*} 
\begin{split}
&  \frac{d}{dt} II^{\calL} \le C\left((\nu t)^{\frac{1}{2}} \nrm{\nabla \omega^{\calL}_{n,0}}_{L^\infty} \nrm{  \omega^{\calL}_{n,0}}_{L^2} (II^{\calL})^{\frac{1}{2}}  +  \nu \nrm{\nabla \omega^{\calL}_{n,0}}^2_{L^2}  \right) \calE^{2\dlt}. 
\end{split}
\end{equation*} Proceeding as in \ref{subsubsec:inviscid-2}, we deduce that \begin{equation*} 
\begin{split}
& II^{\calL} \le C\left( \frac{B^{\frac{3}{2}}}{A} + t^3A^2 \right)
\end{split}
\end{equation*} where this time, \begin{equation*} 
\begin{split}
& A = \nu^{\frac{1}{2}}  \nrm{\nabla \omega^{\calL}_{n,0}}_{L^\infty} \nrm{  \omega^{\calL}_{n,0}}_{L^2} \calE^{2\dlt},
\end{split}
\end{equation*}\begin{equation*} 
\begin{split}
& B  = \nu \nrm{\nabla \omega^{\calL}_{n,0}}^2_{L^2}    \calE^{2\dlt}.
\end{split}
\end{equation*} We arrive at \begin{equation}\label{eq:large-H1}
\begin{split}
II^\calL \le C \frac{\nu}{ \nrm{\nb \omg^\calL_{n,0}}_{L^\infty} \nrm{\omg^\calL_{n,0}}_{L^2} } \left(  \nrm{\nb\omg^\calL_{n,0}}_{L^2} \calE^{-\dlt} + \dlt \frac{ \nrm{\nb \omg^\calL_{n,0}}_{L^\infty} \nrm{\omg^\calL_{n,0}}_{L^2} }{\nrm{\omg^\calL_{n,0}}_{L^\infty} } \right)^3 \calE^{4\dlt} . 
\end{split}
\end{equation}

\subsubsection{$\dot{H}^1$ for the small scale} \label{subsubsec:inviscid-4}
We now define \begin{equation*} 
\begin{split}
& II^\calS := \nrm{ \omega_n^\calS - \omega_n^{\calS,\nu}}_{L^2}^2 . 
\end{split}
\end{equation*}
Recall that
\begin{equation*}
\begin{split}
&\rd_t  \omega_n^{\calS,\nu} + (u_n^{\calL,\nu}\cdot\nabla)\omega_n^{\calS,\nu} = \nabla u_n^{\calL,\nu} \omega_n^{\calS,\nu}+\nu\Delta \omega_n^{\calS,\nu}, \\
&\rd_t  \omega_n^\calS + (u_n^\calL\cdot\nabla)\omega_n^\calS = \nabla u_n^\calL \omega_n^\calS. 
\end{split}
\end{equation*} 
We have \begin{equation*}
\begin{split}
&\frac{1}{2}\frac{d}{dt} \nrm{ \omega_n^\calS - \omega_n^{\calS,\nu}}_{L^2}^2 + \int (u_n^\calL - u_n^{\calL,\nu})\cdot\nabla\omega_n^\calS \cdot( \omega_n^\calS - \omega_n^{\calS,\nu}) \\
&\qquad = \int \nabla u_n^\calL  (\omega_n^\calS - \omega_n^{\calS,\nu} )\cdot ( \omega_n^\calS - \omega_n^{\calS,\nu}) + \int (\nabla u_n^\calL - \nabla u_n^{\calL,\nu}) \omega_n^{\calS,\nu}\cdot ( \omega_n^\calS - \omega_n^{\calS,\nu})   + \nu \int \lap \omega_n^{\calS,\nu}(\omega_n^{\calS,\nu} - \omega_n^\calS). 
\end{split}
\end{equation*} After some routine massaging, \begin{equation*}
\begin{split}
\frac{1}{2}\frac{d}{dt} \nrm{ \omega_n^\calS - \omega_n^{\calS,\nu}}_{L^2}^2  & \le  C\nrm{\nabla\omega_n^\calS}_{L^\infty} \nrm{ u_n^\calL - u_n^{\calL,\nu}}_{L^2}\nrm{ \omega_n^\calS - \omega_n^{\calS,\nu}}_{L^2}  + (1+C\dlt)\nrm{\rd_1 u_{n,1}^\calL}_{L^\infty} \nrm{ \omega_n^\calS - \omega_n^{\calS,\nu}}_{L^2}^2 \\
&\qquad + C\nrm{\omega_n^{\calS,\nu}}_{L^\infty} \nrm{\nabla u_n^\calL-\nabla u_n^{\calL,\nu}}_{L^2}\nrm{ \omega_n^\calS - \omega_n^{\calS,\nu}}_{L^2}  + C\nu \nrm{\nabla \omega_n^\calS}_{L^2}^2  
\end{split}
\end{equation*}  and we rewrite the above as follows: \begin{equation*} 
\begin{split}
& \frac{d}{dt} II^\calS \le 2(1+C\dlt) \nrm{\rd_1u^\calL_{n,1}}_{L^\infty_{t,x}} II^\calS + C \left( \nrm{\nb \omg_n^\calS}_{L^\infty} (I^\calL)^\frac{1}{2} + \nrm{\omg^{\calS,\nu}_n}_{L^\infty}(II^\calL)^{\frac{1}{2}}  \right)(II^\calS)^{\frac{1}{2}} + C\nu \nrm{\nb\omg_{n}^\calS}_{L^2}^2 .
\end{split}
\end{equation*} To simplify the estimate, we introduce \begin{equation*} 
\begin{split}
& \widetilde{II}^\calS(t):=  \exp(-2t(1+C\dlt) \nrm{\rd_1u^\calL_{n,1}}_{L^\infty_{t,x}} ) II^\calS(t) 
\end{split}
\end{equation*} so that \begin{equation*} 
\begin{split}
&\frac{d}{dt}  \widetilde{II}^\calS(t) \lesssim \left( \nrm{\nb \omg_n^\calS}_{L^\infty} (I^\calL)^\frac{1}{2} + \nrm{\omg^{\calS,\nu}_n}_{L^\infty}(II^\calL)^{\frac{1}{2}}  \right)(\widetilde{II}^\calS)^{\frac{1}{2}} \calE^{- t\nrm{\omg^\calL_{n,0}}_{L^\infty}} + \nu \nrm{\nb\omg_{n}^\calS}_{L^2}^2\calE^{-2t\nrm{\omg^\calL_{n,0}}_{L^\infty}}. 
\end{split}
\end{equation*} Now, from previous bounds, we estimate \begin{equation*} 
\begin{split}
&  \left( \nrm{\nb \omg_n^\calS}_{L^\infty} (I^\calL)^\frac{1}{2} + \nrm{\omg^{\calS,\nu}_n}_{L^\infty}(II^\calL)^{\frac{1}{2}}  \right)\calE^{- t\nrm{\omg^\calL_{n,0}}_{L^\infty}}  \\
& \qquad  \lesssim \nu^\frac{1}{2}\calE^{2\dlt}\left( \left(\nrm{\nb\omg_{n,0}^\calS}_{L^\infty}  + \dlt(1 + \dlt \ln \frac{1}{\bell})\frac{\nrm{\nb\omg^\calL_{n,0}}_{L^\infty}}{\nrm{\omg^\calL_{n,0}}_{L^\infty}} \nrm{\omg^\calS_{n,0}}_{L^\infty} \right) \frac{\dlt^{\frac{1}{2}}}{\nrm{\omg^\calL_{n,0}}_{L^\infty}^{\frac{1}{2}}} \nrm{\omg^\calL_{n,0}}_{L^2}  \right. \\
& \left. \qquad  +  \frac{ \nrm{\omg^\calS_{n,0}}_{L^\infty}  }{ \nrm{\nb \omg^\calL_{n,0}}^{\frac{1}{2}}_{L^\infty} \nrm{\omg^\calL_{n,0}}^{\frac{1}{2}}_{L^2} } \left(  \nrm{\nb\omg^\calL_{n,0}}_{L^2} \calE^{-\dlt} + \dlt \frac{ \nrm{\nb \omg^\calL_{n,0}}_{L^\infty} \nrm{\omg^\calL_{n,0}}_{L^2} }{\nrm{\omg^\calL_{n,0}}_{L^\infty} } \right)^{\frac{3}{2}} \right)  := \nu^{\frac{1}{2}} A
\end{split}
\end{equation*} and \begin{equation*} 
\begin{split}
& \nu \nrm{\nb\omg_{n}^\calS}_{L^2}^2\calE^{-2t\nrm{\omg^\calL_{n,0}}_{L^\infty}} \lesssim \nu \left(\nrm{\nb\omg_{n,0}^\calS}_{L^2} + \dlt(1 + \dlt \ln \frac{1}{\bell})\frac{\nrm{\nb\omg^\calL_{n,0}}_{L^\infty}}{\nrm{\omg^\calL_{n,0}}_{L^\infty}} \nrm{\omg^\calS_{n,0}}_{L^2} \right)^2\calE^{2\dlt} =: \nu  B. 
\end{split}
\end{equation*} We compare $ \widetilde{II}^\calS(t)$ with $\nu X(t)$ solving \begin{equation*} 
\begin{split}
& \frac{d}{dt} X(t) = A (X(t))^{\frac{1}{2}} +  B , \quad X(0) = 0. 
\end{split}
\end{equation*} We easily obtain that \begin{equation*} 
\begin{split}
&  \widetilde{II}^\calS(t) \le  \nu X(t) \le C\nu (\frac{B^2}{A^2} + A^2 t^2) 
\end{split}
\end{equation*} and hence \begin{equation*} 
\begin{split}
& {II}^\calS(t) \le   C\nu (\frac{B}{A} + \frac{A\dlt}{\nrm{\omg^\calL_{n,0}}_{L^\infty}})^2 \calE^{2\dlt}. 
\end{split}
\end{equation*} We keep the expressions $A, B$ as they are for now and simplify later with our explicit choice of parameters.

\subsubsection{Final estimate and Proof of Theorem \ref{thm:main4}}\label{subsubsec:final}

In this section, we complete the proof of Theorem \ref{thm:main4}. We proceed in several steps.

\medskip

\noindent \textit{Step 1. Inviscid limit holds for the $L^2$ of the vorticity.}

\medskip

\noindent We would like to have \begin{equation}\label{eq:inv-final}
\begin{split}
\frac{1}{t_n} \int_0^{t_n} \nrm{\omega_n^\calS(t)}_{L^2}^2 dt \gg \frac{1}{t_n}\int_0^{t_n} II^\calS(t) dt,
\end{split}
\end{equation}  {where we recall that \begin{equation*}
	\begin{split}
	II^\calS = \nrm{ \omega_n^\calS - \omega_n^{\calS,\nu}}_{L^2}^2 .
	\end{split}
	\end{equation*} Now recalling that \begin{equation*}
	\begin{split}
	\calE = \exp\left( (1 + C\dlt) \frac{\nrm{\rd_1 u_{n,1}^\calL}_{L^\infty_{t,x}}}{\nrm{\omega^\calL_0}_{L^\infty}} \right),
	\end{split}
	\end{equation*} w}e bound the right hand side simply by \begin{equation*} 
\begin{split}
& \frac{1}{t_n}\int_0^{t_n} II^\calS(t) dt \le \sup_{t \in [0,t_n]} II^\calS(t) \lesssim \nu  (\frac{B}{A} + \frac{A\dlt}{\nrm{\omg^\calL_{n,0}}_{L^\infty}})^2 \calE^{2\dlt}. 
\end{split}
\end{equation*} 
\noindent In this case, a lower bound on the  left hand side is given by \begin{equation*}
\begin{split}
\frac{1}{t_n} \int_0^{t_n} \nrm{\omega_n^{\calS,\nu}(t)}_{L^2}^2 dt \ge \frac{1}{t_n} \int_0^{t_n} \nrm{\omega_{n,1}^{\calS,\nu}(t)}_{L^2}^2 dt \gtrsim \calE^{2\dlt - c\dlt^2} \frac{\nrm{\omega_{n,0}^\calS}_{L^2}^2}{\dlt} .
\end{split}
\end{equation*} The above bound follows from Lemma \ref{lem:LLD} and the Cauchy formula (applied to the first component, recalling that $\omega_{n,0,2} = 0$) \begin{equation*} 
\begin{split}
& \omega_{n,1}^\calS(t,\eta(t,x))  = \rd_1\eta_1(t,x) \omega_{n,0,1}^\calS(x). 
\end{split}
\end{equation*} 
This determines the maximal value of $\nu = \nu_n$ which allows for the crucial estimate \eqref{eq:inv-final}: namely, \begin{equation}\label{nu} 
\begin{split}
& \nu_n := \frac{c}{\dlt} \frac{1}{  (\frac{B}{A} + \frac{A\dlt}{\nrm{\omg^\calL_{n,0}}_{L^\infty}})^2 } \nrm{\omg^\calS_{n,0}}_{L^2}^2. 
\end{split}
\end{equation}


\medskip

\noindent \textit{Step 2. The expression for $\nu_n$.}

\medskip

\noindent Let us now extract the main terms in \eqref{nu}, with our explicit choice of $\omega^\calL_{n,0}$ and $\omega^\calS_{n,0}$. We start by recalling that  \begin{equation*} 
\begin{split}
& A:=  \calE^{2\dlt}\left( \left(\nrm{\nb\omg_{n,0}^\calS}_{L^\infty}  + \dlt(1 + \dlt \ln \frac{1}{\bell})\frac{\nrm{\nb\omg^\calL_{n,0}}_{L^\infty}}{\nrm{\omg^\calL_{n,0}}_{L^\infty}} \nrm{\omg^\calS_{n,0}}_{L^\infty} \right) \frac{\dlt^{\frac{1}{2}}}{\nrm{\omg^\calL_{n,0}}_{L^\infty}^{\frac{1}{2}}} \nrm{\omg^\calL_{n,0}}_{L^2}  \right. \\
& \left. \qquad  +  \frac{ \nrm{\omg^\calS_{n,0}}_{L^\infty}  }{ \nrm{\nb \omg^\calL_{n,0}}^{\frac{1}{2}}_{L^\infty} \nrm{\omg^\calL_{n,0}}^{\frac{1}{2}}_{L^2} } \left(  \nrm{\nb\omg^\calL_{n,0}}_{L^2} \calE^{-\dlt} + \dlt \frac{ \nrm{\nb \omg^\calL_{n,0}}_{L^\infty} \nrm{\omg^\calL_{n,0}}_{L^2} }{\nrm{\omg^\calL_{n,0}}_{L^\infty} } \right)^{\frac{3}{2}} \right)  . 
\end{split}
\end{equation*} We compute 
\begin{equation}\label{eq:norm1} 
\begin{split}
& \nrm{\omg^\calL_{n,0}}_{L^\infty} = 1,\quad \nrm{\omg^\calL_{n,0}}_{L^2} \approx L, \quad  
\nrm{\nb\omg^\calL_{n,0}}_{L^2}\approx \ell^{-\frac{1}{2}}L^{\frac{1}{2}},\quad \nrm{\nb\omg^\calL_{n,0}}_{L^\infty} \approx \ell^{-1}, \quad \|u^{\mathcal L}_{n,0}\|_{L^2}\approx L^{2}.
\end{split}
\end{equation}  
We have, with a free parameter $q \in \bbR$ to be determined,
\begin{equation}\label{eq:norm2}  
\begin{split}
& \nrm{\omg^\calS_{n,0}}_{L^2} = \tilde\ell^{-\frac{2}{q}}, \quad \nrm{\nb\omg^\calS_{n,0}}_{L^2} \approx \tell^{-1-\frac{2}{q}},
\quad
\|u^{\mathcal S}_{n,0}\|_{L^2}\approx \tilde\ell^{1-\frac{2}{q}}, \quad \nrm{\omg^\calS_{n,0}}_{L^\infty} \approx \tell_n^{-1-\frac{2}{q}}, \quad \nrm{\nb\omg^\calS_{n,0}}_{L^\infty} \approx \tell_n^{-2-\frac{2}{q}}. 
\end{split}
\end{equation} Recall that $\tell_n = \ell_n^{1+c\dlt}$ and $\bell_n = \ell L^{-1}:= 2^{-n}$. Now we observe that \begin{equation*} 
\begin{split}
& \tell_n^{-1} \approx \frac{ \nrm{\nb\omg_{n,0}^\calS}_{L^\infty} }{ \nrm{\omg^\calS_{n,0}}_{L^\infty}} \gg  \dlt(1 + \dlt \ln \frac{1}{\bell_n})\frac{\nrm{\nb\omg^\calL_{n,0}}_{L^\infty}}{\nrm{\omg^\calL_{n,0}}_{L^\infty}} \approx  \dlt(1 + \dlt \ln \frac{1}{\bell_n})\ell_n^{-1}
\end{split}
\end{equation*} since $L^{-1} \ge 1 \gg \bell $. 
Note that the above estimate is independent of $L$. Next, it is not difficult to see that (recalling $\calE \gg 1$)
\begin{equation*} 
\begin{split}
& \frac{\nrm{\nb\omg^\calL_{n,0}}_{L^2}}{\nrm{\omg^\calL_{n,0}}_{L^2} } \calE^{-\dlt} \ll   \dlt \frac{ \nrm{\nb \omg^\calL_{n,0}}_{L^\infty} }{\nrm{\omg^\calL_{n,0}}_{L^\infty} }.
\end{split}
\end{equation*} These observations simplify $A$ significantly: \begin{equation*} 
\begin{split}
 A &\approx \calE^{2\dlt} \left(   \frac{\dlt^{\frac{1}{2}}\nrm{\nb\omg_{n,0}^\calS}_{L^\infty} \nrm{\omg^\calL_{n,0}}_{L^2} }{\nrm{\omega^\calL_{n,0}}^{\frac{1}{2}}_{L^\infty}}   +  \frac{ \nrm{\omg^\calS_{n,0}}_{L^\infty}  }{ \nrm{\nb \omg^\calL_{n,0}}^{\frac{1}{2}}_{L^\infty} \nrm{\omg^\calL_{n,0}}^{\frac{1}{2}}_{L^2} } \left(  \dlt \frac{ \nrm{\nb \omg^\calL_{n,0}}_{L^\infty} \nrm{\omg^\calL_{n,0}}_{L^2} }{\nrm{\omg^\calL_{n,0}}_{L^\infty} } \right)^{\frac{3}{2}} \right) \\
&\approx \calE^{2\dlt}\left( \dlt^\frac{1}{2} \tell^{-2-\frac{2}{q}} L + \dlt^{\frac{3}{2}} \tell^{-1-\frac{2}{q}} \ell^{-1}L \right) \\
&\approx \dlt^{\frac{1}{2}}\calE^{2\dlt} \tell^{-2-\frac{2}{q}}L.
\end{split}
\end{equation*} Next, we similarly obtain that \begin{equation*} 
\begin{split}
& B 
 \approx (\tell^{-1-\frac{2}{q}})^2\calE^{2\dlt}.
\end{split}
\end{equation*} Then we can see that \begin{equation*} 
\begin{split}
& \frac{A\dlt}{\nrm{\omg^\calL_{n,0}}_{L^\infty}}
  \gg 
\frac{B}{A}.
\end{split}
\end{equation*} Recalling that $\calE \approx \bell^{-\frac{2}{\pi}(1+C\dlt)}$, we have the following formula for $\nu_n$: \begin{equation}\label{eq:nu_n}
\begin{split}
\nu_n = \frac{c}{\dlt^4} \bell^{4c_0\dlt(1-C\dlt)} \ell^{4(1+C\dlt)} L^{-2}
\end{split}
\end{equation}
which is independent of $q$. Rewriting in terms of $\bell$ and $L$ using $\ell = \bell L$, it is easy to see that $\nu_n\rightarrow 0$ if $L\le 1$. 

\medskip

\noindent \textit{Step 3. Enhanced dissipation.}

\medskip

\noindent  Given our definition of $\nu_n$, we would like to have, with $t_n = \frac{\dlt}{\nrm{\omega_{n}^\calL}_{L^\infty}}$, \begin{equation*}
\begin{split}
\nu_n^{\bar{a}_0} \frac{1}{t_n}\int_0^{t_n} \nrm{\omega_n^\calS(t)}_{L^2}^2 dt  \gtrsim \nrm{u_{n,0}}_{L^2}^2. 
\end{split}
\end{equation*} We compute: 
\begin{equation} \label{avoid-trivial-zeroth-law}
\begin{split}
&\nu_n^{\bar{a}_0} \frac{1}{t_n}\int_0^{t_n} \nrm{\omega_n^\calS(t)}_{L^2}^2 dt  \gtrsim_\dlt \bell^{4\bar{a}_0(1+C\dlt)} L^{\bar{a}_0 (2+C\dlt)} \bell^{-2c_0\dlt(1-C\dlt)} \nrm{\omg_{n,0}^\calS}_{L^2}^2. 
\end{split}
\end{equation}
 We require that the above satisfies $\gtrsim \nrm{u_{n,0}^\calS}_{L^2}^2$ and $\gtrsim \nrm{u_{n,0}^\calL}_{L^2}^2$. For the former, we need \begin{equation*}
\begin{split}
\bell^{4\bar{a}_0(1+C\dlt)} \bell^{-2c_0\dlt(1-C\dlt)} L^{\bar{a}_0 (2+C\dlt)} \gtrsim \tell^2 = \bell^{2(1+C\dlt)}L^{2(1+C\dlt)}. 
\end{split}
\end{equation*} This requirement sets restriction on $L$: \begin{equation*}
\begin{split}
\bell^{ \frac{(2\bar{a}_0-1)(1+C_1\dlt)-c_0\dlt}{(1-\bar{a}_0)(1+C_2\dlt)} } \gtrsim L. 
\end{split}
\end{equation*} From the above, one sees that there are two cases: $a_0 \le \frac{1}{2}$ and $1>a_0> \frac{1}{2}$. In the former, the left hand side satisfies $\gg 1$, so we simply fix $L = 1$ for all $n$. In the latter, we simply define \begin{equation*}
\begin{split}
L = \bell^{\gamma },\quad \gamma:=  \frac{(2\bar{a}_0-1)(1+C_1\dlt)-c_0\dlt}{(1-\bar{a}_0)(1+C_2\dlt)}
\end{split}
\end{equation*} where $C_1,C_2$, and $c_0$ are some positive absolute constants. Note that $L \ll 1$. In the following, we shall proceed with assuming $1>a_0>\frac{1}{2}$. Now, for $\gtrsim \nrm{u_{n,0}^\calL}_{L^2}^2$, we need \begin{equation*}
\begin{split}
\bell^{4\bar{a}_0(1+C\dlt)} \bell^{-2c_0\dlt(1-C\dlt)} L^{\bar{a}_0 (2+C\dlt)} \tell^{-\frac{4}{q}} \gtrsim L^2
\end{split}
\end{equation*} and this determines the value of $q$. We can just require that $\nrm{u^\calS_{n,0}}_{L^2}^2 {\approx} L^2$; \begin{equation*}
\begin{split}
\bell^{ (1+\gamma)(1+C_3\dlt)(1-\frac{2}{q}) }=\tell^{1-\frac{2}{q}}  {\approx} L =  \bell^\gamma 
\end{split}
\end{equation*} which gives \begin{equation*}
\begin{split}
\frac{\gamma}{(1+\gamma)(1+C_3\dlt)} {=} 1 - \frac{2}{q}. 
\end{split}
\end{equation*} Note that {$q = 2\left(1-\frac{\gamma}{(1+\gamma)(1+C_3\delta)}\right)^{-1}$} clearly satisfies the above, and in this case we have that $\nrm{u^\calS_{n,0}}_{L^2} \approx {L^2}$. Finally, we note that when $a_0\le \frac{1}{2}$, we already have the lower bound on the energy from the large-scale: \begin{equation*}
\begin{split}
\nrm{u^\calL_{n,0}}_{L^2}^2 \gtrsim 1, 
\end{split}
\end{equation*} so that we can take $q$ in a way that $\nrm{\omg^\calS_{n,0}}_{L^2}$ is uniformly bounded in $n$. The proof is now complete. \qedsymbol

\section{Conclusion}

We prepared small-scale vortex blob and large-scale anti-parallel
vortex tubes for the initial data, and showed that the corresponding
3D Navier-Stokes flow creates instantaneous vortex-stretching. {In turn, using this stretching, we showed that the  flows satisfy a modified version of the zeroth law in a uniform time interval which in particular implies enhanced dissipation.} 

\bibliographystyle{amsplain}


\end{document}